\documentclass[12pt,a4paper,reqno]{amsart}
\usepackage{amsmath,amssymb,fullpage}
\usepackage{amsthm}
\usepackage{mathrsfs}
\usepackage{esint}  
\usepackage{color}
\usepackage{graphicx}
\usepackage{psfrag}
\usepackage{enumerate}
\usepackage{moreverb}

\def\qest{q_{\rm est}}
\def\qctr{q_{\rm ctr}}

\def \qdoe{q_{\rm opt}}
\def\Cest{C_{\rm est}}
\def\CinvV{C_{\rm inv}}
\def\Cinv{C_{\rm inv}}

\def\Cmark{C_{\rm mark}}
\def\Cstab{C_{\rm stab}}

\def\Cone{C_1}
\def\Ctwo{C_2}
\def\H{\widetilde{H}}

\def\N{\mathbb{N}}
\def\R{\mathbb{R}}
\def\A{\mathbb{A}}

\def\T{\mathbb{T}}

\def\Z{\mathbb{Z}}

\def\EE{\mathcal{E}}

\def\HH{\mathcal{H}}

\def\KK{\mathcal{K}}
\def\WW{\mathcal{W}}
\def\NN{\mathcal{N}}
\def\RR{\mathcal{R}}
\def\SS{\mathcal{S}}
\def\TT{\mathcal{T}}
\def\XX{\mathcal{X}}

\def\refine{{\tt ref}}

\def\diam{{\rm diam}}

\def\supp{{\rm supp}}

\newcommand{\set}[3][\big]{#1\{#2\,:\,#3#1\}}
\newcommand{\norm}[3][]{#1\|#2#1\|_{#3}}
\def\dual#1#2{\langle#1\,;\,#2\rangle}

\numberwithin{equation}{section}
\numberwithin{figure}{section}
\newtheorem{theorem}{Theorem}[section]
\newtheorem{proposition}[theorem]{Proposition}
\newtheorem{lemma}[theorem]{Lemma}

\newtheorem{algorithm}[theorem]{Algorithm}
\newtheorem{remark}[theorem]{Remark}
\newtheorem{assumption}[theorem]{Assumption}

\renewcommand{\subsection}[1]{\refstepcounter{subsection}\medskip{\bf\thesubsection.~#1.}}

\newcounter{const}

\date{\today}

\def\MM{\mathcal M}
\def\XX{\mathcal X}
\def\Crel{C_{\rm rel}}

\def\Clin{C_{\rm lin}}

\def\Cstb{C_{\rm sz}}
\def\Chdown{C_{\rm wt}^{-1}}
\def\Chup{C_{\rm wt}}
\def\qlin{q_{\rm lin}}
\def \Cdrel{C_{\rm rel}}
\def \Cref{C_{\rm ref}}
\def \Cmon{C_{\rm mon}}
\def \Cmesh{C_{\rm mesh}}

\title{Optimal convergence\\ for adaptive IGA boundary element methods\\ for weakly-singular integral equations}

\author{Michael Feischl, Gregor Gantner, Alexander Haberl, Dirk Praetorius}

\begin{document}




\begin{abstract}
In a recent work~\cite{resigabem}, we analyzed a weighted-residual error estimator for isogeometric boundary element methods in 2D and proposed an adaptive algorithm which steers the local mesh-refinement of the underlying partition as well as the multiplicity of the knots. In the present work, we give a mathematical proof that this algorithm leads to convergence even with optimal algebraic rates. 
Technical contributions include a novel mesh-size function which also monitors the knot multiplicity as well as inverse estimates for NURBS in fractional-order Sobolev norms.
\end{abstract}


\maketitle

\section{Introduction} 

\subsection{Isogeometric analysis} The central idea of isogeometric analysis (IGA) is to use the same ansatz functions for the discretization of the partial differential equation at hand, as are used for the representation of the problem geometry. Usually, the problem geometry $\Omega$ is represented in CAD by means of non-uniform rational B-splines (NURBS), T-splines, or hierarchical splines. This concept, originally invented in~\cite{hughes2005} for finite element methods (IGAFEM) has proved very fruitful in applications; see also the monograph \cite{bible}. 

Since CAD directly provides a parametrization of the boundary $\partial \Omega$, this makes the boundary element method (BEM) the most attractive numerical scheme, if applicable (i.e., provided that the fundamental solution of the differential operator is explicitly known).
 However, compared 
to the IGAFEM literature, only little is found for isogeometric BEM (IGABEM). 
The latter has first been considered for 2D BEM in \cite{igabem2d} and for 3D BEM in \cite{igabem3d}.
Unlike standard BEM with piecewise polynomials which is well-studied in the literature, cf.~the monographs~\cite{ss,steinbach} and the references therein,  the numerical analysis of IGABEM is widely open. We refer {to~\cite{SBTR,helmholtziga,TM,zechner}} for numerical experiments, {to \cite{zechnerH} for fast IGABEM with $\mathcal{H}$-matrices}, and to~\cite{stokesiga} for some quadrature analysis.
To the best of our knowledge, 
{\sl {\sl a~posteriori}} error estimation for IGABEM, however, has only been considered for simple 2D model problems in the recent own works~\cite{igafaermann,resigabem}. 
The present work extends the techniques from standard BEM to non-polynomial ansatz functions.
 {The remarkable flexibility of the IGA ansatz functions to manipulate their smoothness properties motivates the development of a new adaptive algorithm which does not only automatically adapt the mesh-width, but also the continuity of the IGA ansatz function to exploit the additional freedomss and the full potential of IGA.
This is the first algorithm which simultaneously steers the resulution and the smoothness of the ansatz functions, and, it  may thus be a first step to a full $hpk$-adaptive algorithm.}

For standard BEM with discontinuous piecewise polynomials, {\sl {\sl a~posteriori}} error estimation and adaptive mesh-refinement are well understood. We refer to~\cite{cmps,cms,meshoneD} for weighted-residual error estimators and to~\cite{arcme,zzbem} for  recent overviews on available {\sl a~posteriori} error estimation strategies. Moreover, optimal convergence of mesh-refining adaptive algorithms has recently been proved for polyhedral boundaries~\cite{part1,part2,fkmp} as well as smooth boundaries~\cite{gantumur}. The  work~\cite{invest} allows to transfer these results to piecewise smooth boundaries; see also the discussion in the review article~\cite{axioms}.

While this work focusses on adaptive IGABEM, adaptive IGAFEM is considered, e.g., in \cite{juettler1,juettler2}.
A rigorous error and convergence analysis in the frame of adaptive IGAFEM is first  found in \cite{buffa} which  proves linear convergence for some adaptive IGAFEM with hierarchical splines for the Poisson equation, and optimal rates are announced for some future work.

\subsection{Model problem} We develop and analyze an adaptive algorithm for the following model problem:  Let $\Omega\subset\R^2$ be a Lipschitz domain with $\diam(\Omega)<1$ and $\Gamma\subseteq \partial\Omega$ be a compact, piecewise smooth part of its boundary with finitely many connected components.
We consider the weakly-singular boundary integral equation 
\begin{align}\label{eq:strong}
 V\phi(x):=-\frac{1}{2\pi}\int_\Gamma\log|x-y|\,\phi(y)\,dy = f(x)
 \quad\text{for all }x\in\Gamma_{},
\end{align}
where the right-hand side $f$ is given and the density $\phi$ is sought. 
{We note that \eqref{eq:strong} for $\Gamma=\partial\Omega$ is equivalent to the Laplace-Dirichlet problem
\begin{align}
-\Delta u=0\text{ in }\Omega\quad\text{with } u=f\text{ on }\Gamma,\quad\text{where }u:=V\phi.
\end{align}}
To approximate $\phi$, we employ a Galerkin boundary element method (BEM) with ansatz spaces consisting of $p$-th order NURBS. The convergence order for uniform partitions of $\Gamma$ is usually suboptimal, since the unknown density $\phi$ may exhibit singularities, which are stronger than the singularities in the geometry. In~\cite{resigabem}, we analyzed a weighted-residual error estimator and proposed an adaptive algorithm which uses this {\sl a~posteriori} error information to steer the $h$-refinement of the underlying partition as well as the local smoothness of the NURBS across the nodes of the adaptively refined partitions. It reflects the fact that it is {\sl a~priori} unknown, where the singular and smooth parts of the density $\phi$ are located and where approximation by nonsmooth resp. smooth functions is required. In~\cite{resigabem}, we observed experimentally that the proposed algorithm detects singularities and possible jumps of $\phi$ and leads to optimal convergence behavior. In particular, we observed that the proposed adaptive strategy is also superior to adaptive BEM with discontinuous piecewise polynomials in the sense that our adaptive NURBS discretization requires less degrees of freedom to reach a prescribed accuracy.


\subsection{Contributions}
We  prove that the adaptive algorithm from~\cite{resigabem} is rate optimal in the sense of~\cite{axioms}:
Let $\mu_\ell$ be the weighted-residual error estimator in the $\ell$-th step of the adaptive algorithm. First, the adaptive algorithm leads to linear convergence of the error estimator, i.e., $\mu_{\ell+n} \le Cq^n\mu_\ell$ for all $\ell,n\in\N_0$ and some independent constants $C>0$ and $0<q<1$. Moreover, for sufficiently small marking parameters, i.e. agressive adaptive refinement, the estimator decays even with the optimal algebraic convergence rate. Here, the important innovation is that the adaptive algorithm does not only steer the local refinement of the underlying partition (as is the case in the available  literature, e.g., \cite{axioms,part1,part2,fkmp,gantumur}), but also the multiplicity of the knots. 
{In particular, the present work is  the first available optimality result for adaptive algorithms in the frame of isogeometric analysis.} Additionally, we can prove at least plain convergence if the adaptive algorithm is driven by the Faermann estimator $\eta_\ell$ analyzed in \cite{igafaermann} instead of the weighted-residual estimator $\mu_\ell$, which generalizes a corresponding result for standard adaptive BEM  \cite{mitscha}.

Technical contributions of general interest include a novel mesh-size function $h\in L^\infty(\Gamma)$ which is locally equivalent to the element length (i.e., $h|_T \simeq {\rm length}(T)$ for all elements $T$), but also accounts for the knot multiplicity. Moreover, for $0<\sigma<1$we prove a local inverse estimate $\norm{h^{\sigma}\Psi}{L^2(\Gamma)} \le C\,\norm{\Psi}{\H^{-\sigma}(\Gamma)}$ for NURBS on locally refined meshes. 
Similar estimates for piecewise polynomials are shown in \cite{inversefaermann,inversesauter,georgoulis}, while \cite{approximation} considers NURBS but integer-order Sobolev norms only. 

Throughout, all results apply for piecewise smooth parametrizations $\gamma$ of $\Gamma$ and discrete NURBS spaces. In particular, the analysis thus covers the NURBS ansatz used for IGABEM, where the same ansatz functions are used for the discretization of the integral equation and for the resolution of the problem geometry,  as well as spline spaces and even piecewise polynomials on the piecewise smooth boundary $\Gamma$ which can be understood as special NURBS.

\subsection {Outline}
The remainder of this work is organized as follows: Section~\ref{section:preliminaries} fixes the notation and provides the necessary preliminaries. This includes, e.g., the involved Sobolev spaces (Section~\ref{section:sobolev}), the functional analytic setting of the weakly-singular integral equation (Section~\ref{section:weaksing}), the assumptions on the parametrization of the boundary $\Gamma$ (Section~\ref{subsec:boundary parametrization}), the discretization of the boundary (Section~\ref{section:boundary:discrete}), the mesh-refinement strategy (Section~\ref{section:mesh-refinement}), B-splines and NURBS (Section~\ref{subsec:splines}), and the IGABEM ansatz spaces (Section~\ref{section:igabem}). Section~\ref{section:algorithm} states our adaptive algorithm (Algorithm~\ref{the algorithm}) from~\cite{resigabem} and formulates the main theorems on linear convergence with optimal rates  for the weighted-residual estimator $\mu_\ell$ (Theorem~\ref{thm:main}) and on plain convergence for the Faermann estimator $\eta_\ell$ (Theorem~\ref{thm:faermann}). 
The linear convergence for the $\mu_\ell$-driven algorithm is proved in Section~\ref{section:rlinear}. The proof requires an inverse estimate for NURBS in a fractional-order Sobolev norm  (Proposition~\ref{prop:inverse estimate}) as well as a novel mesh-size function for B-spline and NURBS discretizations (Proposition~\ref{lem:mesh function}) which might be of independent interest. The proof of optimal convergence behaviour is given in Section~\ref{section:optimal}. 
In Section~\ref{section:faermann}, we show convergence for the $\eta_\ell$-driven algorithm.

For the empirical verification of the optimal convergence behavior of Algorithm~\ref{the algorithm} for $\mu_\ell$- as well as $\eta_\ell$-driven adaptivity and a comparison of IGABEM and standard BEM with discontinuous piecewise polynomials, we refer to the numerous numerical experiments in our preceding work~\cite{resigabem}.

\section{Preliminaries}
\label{section:preliminaries}



\subsection{General notation}
Throughout, $|\cdot|$ denotes the absolute value of scalars, the Euclidean norm of vectors in $\R^2$, the measure of a set in $\R$ (e.g., the length of an interval), or the arclength of a curve in $\R^{2}$.
The respective meaning will be clear from the context.
We write $A\lesssim B$ to abbreviate $A\le cB$ with some generic constant $c>0$ which is clear from the context.
Moreover, $A\simeq B$ abbreviates $A\lesssim B\lesssim A$. 
Throughout, mesh-related quantities have the same index, e.g., $\NN_\star$ is the set of nodes of the partition $\TT_\star$, and $h_\star$ is the corresponding local mesh-width etc. The analogous notation is used for partitions $\TT_+$ resp.\ $\TT_\ell$ etc.

\def\Cgamma{C_\Gamma}
\subsection{Sobolev spaces}
\label{section:sobolev}
For any measurable subset ${\Gamma_0}\subseteq\Gamma$, let $L^2({\Gamma_0})$ denote
the Lebesgue space of all square integrable functions which is associated with the norm
$\norm{u}{L^2({\Gamma_0})}^2:=\int_{\Gamma_0} |u(x)|^2\,dx$.
We define  for any $0<\sigma\le 1$ the Hilbert space
\begin{align}
 H^{\sigma}({\Gamma_0}) := \set{u\in L^2({\Gamma_0})}{\norm{u}{H^{\sigma}({\Gamma_0})}<\infty},
\end{align}
associated with the Sobolev-Slobodeckij norm
\begin{align}\label{eq:SS-norm}
 \norm{u}{H^{\sigma}({\Gamma_0})}^2
 := \norm{u}{L^2({\Gamma_0})}^2
 + |u|_{H^{\sigma}({\Gamma_0})}^2,
 \end{align}
 {with
 \begin{align}
 |u|_{H^{\sigma}({\Gamma_0})}^2 := \begin{cases} \int_{\Gamma_0}\int_{\Gamma_0}\frac{|u(x)-u(y)|^2}{|x-y|^{1+2\sigma}}\,dy\,dx, & \text{for }0<\sigma<1\\\norm{\partial_\Gamma u}{L^2(\Gamma_0)},&\text{for }\sigma=1,\end{cases}
 \end{align}
 where $\partial_\gamma$ denotes the arclength derivative.
 }
For finite intervals $I\subseteq \R$, we use analogous definitions.
By $\H^{-\sigma}({\Gamma_0})$, we denote the dual space of $H^{\sigma}({\Gamma_0})$, where duality
is understood with respect to the extended $L^2({\Gamma_0})$-scalar product, i.e.,
\begin{align}
 \dual{u}{\phi}_{\Gamma_0} = \int_{\Gamma_0} u(x)\phi(x)\,dx
 \quad\text{for all }u\in H^{\sigma}({\Gamma_0})
 \text{ and }\phi\in L^2({\Gamma_0}).
\end{align}
{We note that $H^{\sigma}(\Gamma)\subset L^2(\Gamma)\subset \H^{-\sigma}(\Gamma)$ form a Gelfand triple and all inclusions are dense and compact.}
Amongst other equivalent definitions of $H^{\sigma}({\Gamma_0})$ are 
for example
interpolation techniques.
All these definitions provide the same space of functions but different norms, where
norm equivalence constants depend only  on ${\Gamma_0}$; see, e.g., the monographs~\cite{hw,mclean}
and the references therein.
Throughout our proofs, we shall use the Sobolev-Slobodeckij
norm~\eqref{eq:SS-norm}, since it is  numerically computable.

\subsection{{Weakly-singular integral equation}}
\label{section:weaksing}
It is known~\cite{hw,mclean} that the weakly-singular integral operator $V:\H^{-1/2}(\Gamma)\to H^{1/2}(\Gamma)$ from~\eqref{eq:strong} is a symmetric and elliptic isomorphism if $\diam(\Omega)<1$ which can always be achieved by scaling. For a given right-hand side $f\in H^{1/2}(\Gamma)$,
the strong form~\eqref{eq:strong} is thus equivalently stated by 
\begin{align}\label{eq:weak}
\dual{V\phi}{\psi}_{\Gamma}
 = \dual{f}{\psi}_{\Gamma}
 \quad\text{for all }\psi\in\H^{-1/2}(\Gamma_{}),
\end{align}
and the left-hand side defines an equivalent scalar product on $\H^{-1/2}(\Gamma)$.
In particular, the  Lax-Milgram lemma proves existence and uniqueness of the solution $\phi\in\H^{-1/2}(\Gamma)$. Additionally, $V:L^2(\Gamma)\to H^1(\Gamma)$ is well-defined, linear,  and continuous.

In the Galerkin boundary element method, the test
space $\H^{-1/2}(\Gamma_{})$ is replaced by some discrete subspace  $\XX_{\star}\subset {L^{2}(\Gamma_{})}\subset\H^{-1/2}(\Gamma_{})$.
Again, the Lax-Milgram lemma guarantees existence and uniqueness of the solution
$\Phi_\star\in\XX_\star$ of the discrete variational formulation 
\begin{align}\label{eq:discrete}
\dual{V\Phi_\star}{\Psi_\star}_{\Gamma}
 = \dual{f}{\Psi_\star}_{\Gamma}
 \quad\text{for all }\Psi_\star\in\XX_\star.
\end{align}
Below, we shall assume that $\XX_\star$ is linked to a partition $\TT_\star$ of $\Gamma$ into a set of connected segments.

\subsection{Boundary parametrization}
\label{subsec:boundary parametrization}
Let $\Gamma = \bigcup_i\Gamma_i$ be decomposed into its finitely many connected
components $\Gamma_i$. Since the $\Gamma_i$ are compact and piecewise
smooth as well, it holds
\begin{align*}
 \|u\|^2_{H^{1/2}(\Gamma)}
 = \sum_{i}\|u\|^2_{H^{1/2}(\Gamma_i)} 
 + \sum_{{i,j}\atop{i\neq j}} \int_{\Gamma_i}\int_{\Gamma_j}\frac{|u(x)-u(y)|^2}{|x-y|^2}\,dy\,dx
 \simeq \sum_{i} \|u\|^2_{H^{1/2}(\Gamma_i)};
\end{align*}
see, e.g., \cite[Section 2.2]{igafaermann}.
The usual piecewise polynomial and NURBS basis functions have connected support
and are hence supported by some \emph{single} $\Gamma_i$ each. Without loss of generality
and for the ease of presentation, 
we may therefore assume {throughout} that $\Gamma$ is connected. All results of this
work remain valid for non-connected $\Gamma$.

We assume that either $\Gamma=\partial\Omega$ is parametrized by a closed continuous and
piecewise {two times} continuously differentiable path $\gamma:[a,b]\to\Gamma$ such
that the restriction $\gamma|_{[a,b)}$ is even bijective, or that $\Gamma\subsetneqq\partial\Omega$ is parametrized by a bijective continuous and piecewise two times continuously differentiable path $\gamma:[a,b]\to\Gamma$.  In the first case, we speak of \textit{closed} $\Gamma=\partial\Omega$, whereas the second case is referred to as \textit{open} $\Gamma\subsetneqq\partial\Omega$.

For closed $\Gamma=\partial\Omega$, we denote the $(b-a)$-periodic extension to $\R$ also by $\gamma$. 
For the left and right derivative of $\gamma$, we assume that {$\gamma^{\prime_\ell}(t)\neq 0$ for $t\in(a,b]$ and $\gamma^{\prime_r}(t)\neq 0$  for $t\in [a,b)$.}
Moreover we assume that $\gamma^{\prime_\ell}(t)
+c\gamma^{\prime_r}(t)\neq0$ for all $c>0$ {and $t\in[a,b]$ resp. $t\in(a,b)$.} 
Finally, let $\gamma_L:[0,L]\to\Gamma$ denote the arclength parametrization, i.e.,
$|\gamma_L^{\prime_\ell}(t)| = 1 = |\gamma_L^{\prime_r}(t)|$, and its periodic extension. Elementary
differential geometry yields bi-Lipschitz continuity
\begin{align}\label{eq:bi-Lipschitz}
 \Cgamma^{-1} \le \frac{|\gamma_L(s)-\gamma_L(t)|}{|s-t|}\le\Cgamma
 \quad\text{for }s,t\in\R, {\text{ with }\begin{cases}
 |s-t|\le \frac{3}{4}\,L, \text{ for closed }\Gamma,\\
  s\neq t\in [0,L], \text{ for open }\Gamma,
\end{cases}}
\end{align}
where $\Cgamma>0$ depends only on $\Gamma$.
A proof is given in \cite[Lemma 2.1]{diplarbeit} for closed $\Gamma=\partial\Omega$. 
For open $\Gamma\subsetneqq\partial\Omega$, the proof is even simpler. 

Let $I\subseteq[a,b]$.
If $\Gamma=\partial\Omega$ is closed and $|I|\le \frac{3}{4} L$ resp.\ if $\Gamma\subsetneqq\partial\Omega$ is open, the bi-Lipschitz continuity~\eqref{eq:bi-Lipschitz} implies
\begin{align}\label{eq:equivalent Hsnorm} 
\Cgamma^{-1}|u\circ\gamma_{L}|_{H^{1/2}(I)}\leq |u|_{H^{1/2}(\gamma_L(I))}\leq \Cgamma|u\circ\gamma_{L}|_{H^{1/2}(I)}.
\end{align}

\subsection{Boundary discretization}
\label{section:boundary:discrete}
In the following, we describe the different quantities which define the discretization.

{\bf Nodes $\boldsymbol{z_j=\gamma(\check{z}_j)\in\mathcal{N}_\star}$.}\quad Let $\mathcal{N}_\star:=\set{z_j}{j=1,\dots,n}$ and $z_0:=z_n$ for closed $\Gamma=\partial\Omega$ resp.\ $\mathcal{N}_\star:=\set{z_j}{j=0,\dots,n}$ for open $\Gamma\subsetneqq \partial\Omega$ be a set of nodes. We suppose that $z_j=\gamma(\check{z}_j)$ for some $\check{z}_j\in[a,b]$ with
$a=\check{z}_0<\check{z}_1<\check{z}_2<\dots<\check{z}_n=b$ such that 
$\gamma|_{[\check{z}_{j-1},\check{z}_j]}\in C^2([\check{z}_{j-1},\check{z}_j])$.

{\bf Multiplicity $\boldsymbol{\#z_j}$ and knots $\boldsymbol{\KK_\star},\boldsymbol{\check{\KK}_\star}$.}\quad
Let $p\in\N_0$ be some fixed polynomial order.
{Each node $z_j$ has a multiplicity $\#z_j\in\{1,2\dots, p+1\}$ with $\#{z}_{0}=\#z_{n}=p+1$}.
This induces knots 
\begin{align}
\KK_\star=(\underbrace{z_k,\dots,z_k}_{\#z_k-\text{times}},\dots,\underbrace{z_n,\dots,z_n}_{\#z_n-\text{times}}),
\end{align} 
with $k=1$ resp. $k=0$ and corresponding knots $\check{\KK}_\star:=\gamma|_{(a,b]}^{-1}(\KK_\star)$ resp. $\check{\KK}_\star:=\gamma^{-1}(\KK_\star)$ on the parameter domain $[a,b]$.

{\bf Elements, partition $\boldsymbol{\mathcal{T}_\star}$, and $\boldsymbol{[T]}$, $\boldsymbol{[\mathcal{T}_\star}]$.} \quad
Let $\mathcal{T}_\star=\{T_1,\dots,T_n\}$ be a partition of $\Gamma$ into compact and connected segments $T_j=\gamma(\check{T}_j)$ with $\check{T}_j=[\check{z}_{j-1},\check{z}_j]$.
We define
\begin{align}
[\TT_\star]:=\set{[T]}{T\in \TT_\star}
\quad\text{with}\quad
[T]:=(T,\#z_{T,1},\#z_{T,2}),
\end{align}
where $z_{T,1}=z_{j-1}$ and $z_{T,2}=z_j$ are the two nodes of $T=T_j$.

{\bf Local mesh-sizes $\boldsymbol{h_{\star,T}}$, $\boldsymbol{\check{h}_{\star,T}}$ and $\boldsymbol{h_\star}$, $\boldsymbol{\check{h}_\star}$.}\quad
The arclength of each element $T\in \mathcal{T}_\star$ is denoted by $h_{\star,T}$.
We define the local mesh-width function $h_\star\in L^\infty(\Gamma)$ by $h_\star|_T=h_{\star,T}$.
Additionally, we define for each element $T\in \mathcal{T}_\star$ its length $\check{h}_{\star,T}:=|\gamma^{-1}(T)|$ with respect to the parameter domain $[a,b]$.
This gives rise to $\check{h}_\star\in L^\infty(\Gamma)$ with $\check{h}_\star|_T=\check{h}_{\star,T}$. 
Note that the lengths $h_{\star,T}$ and $\check{h}_{\star,T}$ of an element $T$ are equivalent, where the equivalence constants depend only on $\gamma$.

{\bf Local mesh-ratios $ \boldsymbol{\check{\kappa}_\star}$.}\quad
We define the {local mesh-ratio} by
\begin{align}\label{eq:meshratio}
\check{\kappa}_\star&:=\max\set{\check{h}_{\star,T}/\check{h}_{\star,T'}}{{T},{T}'\in\TT_\star \text{ with }  T\cap T'\neq \emptyset}.
\end{align}

{\bf Patches $\boldsymbol{\omega_\star(z)}$, $\boldsymbol{\omega_\star(U)}$, $\boldsymbol{\omega_\star(\mathcal{U})}$,  {and $\boldsymbol{\bigcup\mathcal U}$}.}
For each set $U\subseteq\Gamma$, we  {inductively define for $m\in\N_0$}
{\begin{align*}
 \omega_\star^m(U) :=\begin{cases} U\quad&\text{if }m=0,\\
 \omega_\star(U):= \bigcup\set{T\in \TT_\star}{T\cap U\neq \emptyset}\quad&\text{if }m=1,\\
 \omega_\star(\omega_\star^{m-1}(U)) \quad&\text{if }m>1.\end{cases}
\end{align*}}
For nodes $z\in\Gamma$, we abbreviate $\omega_\star(z)=:\omega_\star(\{z\}).$
{For each set $\mathcal{U}\subseteq[\TT_\star]$, we define
\begin{align*}
 \bigcup\mathcal U := \bigcup\set{T\in\TT_\star}{[T]\in\mathcal U},
\end{align*}
and
\begin{align*}
 \omega_\star^m(\mathcal{U}) := \omega_\star^m(\bigcup\mathcal{U}).
\end{align*}}
\begin{figure}[t] 
\psfrag{pacman (Section 5.3)}[c][c]{}
\psfrag{z}[r][r]{z}
\psfrag{T}[r][r]{T}
\psfrag{45}[r][r]{\tiny $45^\circ$}
\begin{center}
\includegraphics[width=0.5\textwidth]{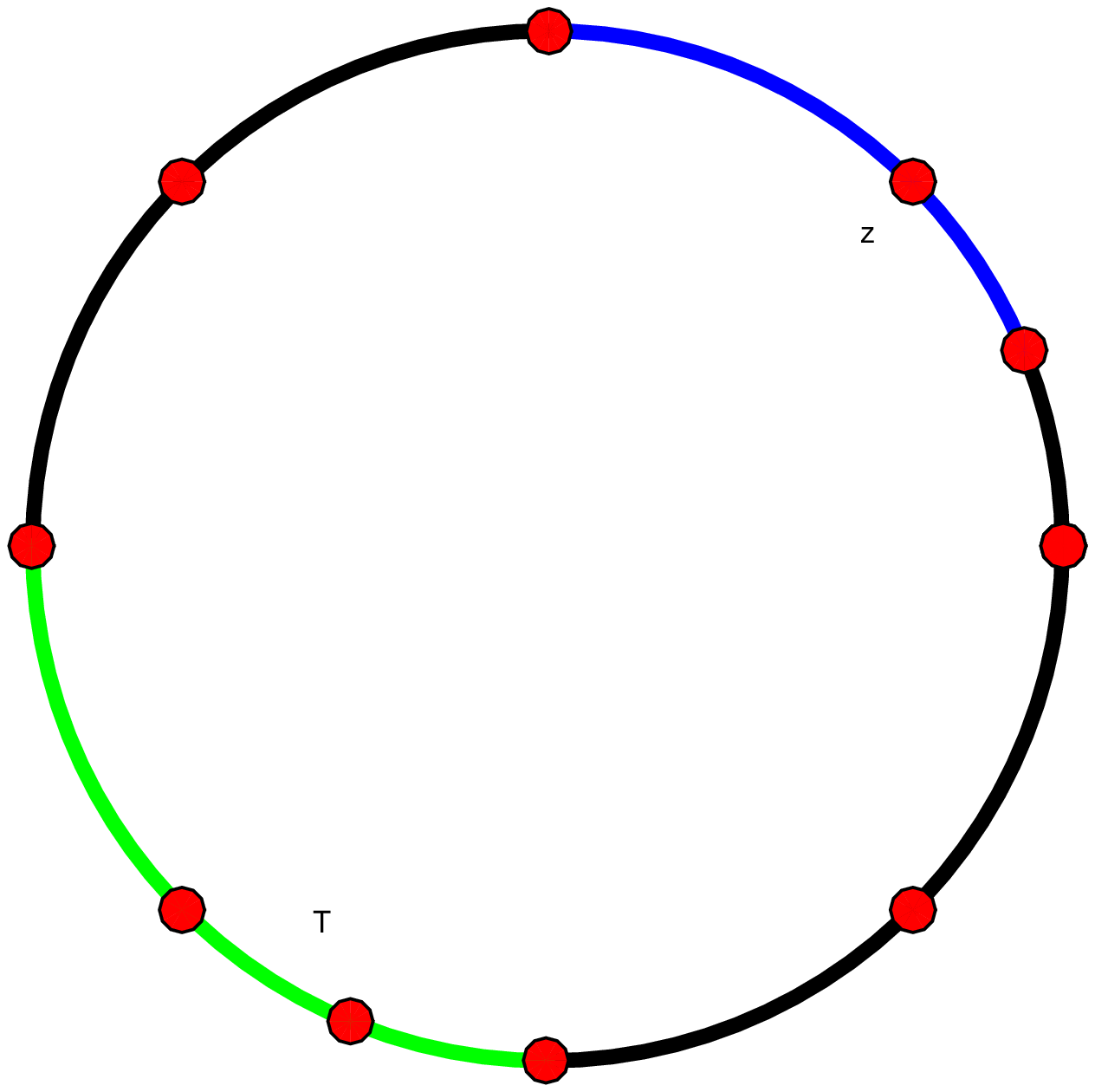}
\end{center}
\caption{The patch $\omega_\star(z)$ of some node $z\in\NN_\star$ resp. the patch $\omega_\star(T)$ are illustrated in blue resp. green.} 
\label{fig:patches}
\end{figure}

\subsection{Mesh-refinement}
\label{section:mesh-refinement}
Suppose that we are given a deterministic mesh-refinement strategy $\refine(\cdot)$ such 
that, for each mesh $[\TT_{\star}]$ and an arbitrary set of marked nodes  $\MM_\star\subseteq \NN_\star$, the application $[\TT_{+}] := \refine([\TT_\star],\MM_\star)$ provides a mesh in the sense of Section~\ref{section:boundary:discrete} such that, first, the marked nodes belong to the union of the refined elements, i.e., 
$\MM_\star\subset\bigcup([\TT_\star]\setminus [\TT_{+}])$, and, second, the knots $\KK_\star$ form a subsequence of the knots $\KK_{+}$.
The latter implies the estimate
\begin{align}
|[\TT_\star]\setminus [\TT_{+}]|\le 2(|\KK_{+}|-|\KK_\star|),
\end{align}
{since $[\TT_\star]\setminus[\TT_+]$ is the set of all elements in which a new knot is inserted and one new knot can be inserted in at most $2$ elements of the old mesh, i.e., at the intersection of $2$ elements.}

We write $[\TT_{+}]\in\refine([\TT_\star])$, if there exist finitely many meshes $[\TT_1],\dots,[\TT_\ell]$ and subsets $\MM_{j}\subseteq\NN_{j}$ of the corresponding nodes such that $[\TT_\star] = [\TT_1]$, $[\TT_{+}]=[\TT_\ell]$, and $[\TT_{j}] = \refine([\TT_{j-1}],\MM_{j-1})$ for all $j=2,\dots,\ell$, where we formally allow
$m=1$, i.e., $[\TT_\star]=[\TT_{1}]\in\refine([\TT_\star])$. 


For the proof of our main result, we need the following assumptions on $\refine(\cdot)$.
\begin{assumption}\label{ass:mesh assumption}
{For an arbitrary initial mesh $[\TT_0]$ and $[\T]:=\refine([\TT_0])$, we assume that the mesh-refinement strategy satisfies the properties {\rm(M1)--(M3)}:}
\begin{itemize}
\item[{\rm (M1)}] There exists a constant $\check{\kappa}_{\max}\ge 1$ such that the {local mesh-ratios}~\eqref{eq:meshratio} are uniformly bounded
\begin{align}\label{eq:kappamax}
\check{\kappa}_\star\le \check{\kappa}_{\max}\quad\text{for all }[\TT_\star]\in [\T].
\end{align}
\item[{\rm (M2)}]
For all $[\TT_\star],[\TT_{+}]\in [\T]$, there is a common refinement $[\TT_\star\oplus \TT_{+}]\in\refine([\TT_\star])\cap\refine([\TT_{+}])$ such that the  knots $\KK_\star\oplus \KK_{+}$ of $[\TT_\star\oplus \TT_{+}]$ satisfy the overlay estimate
\begin{align}
|\KK_\star\oplus \KK_{+}|\le |\KK_\star|+|\KK_{+}|-|\KK_0|.
\end{align}
\item[{\rm (M3)}]
Each sequence $[\TT_\ell]\in [\T]$ of meshes generated by successive mesh-refinement, i.e., $[\TT_j] = \refine([\TT_{j-1}],\MM_{j-1})$ for all $j \in \N$ and arbitrary $\MM_j \subseteq  \NN_j,$  satisfies
\begin{align}
|\KK_\ell|-|\KK_0|\le \Cmesh \sum_{j=0}^{\ell-1} |\MM_j|\quad \text{for }\ell\in \N,
\end{align}
where $\Cmesh>0$ depends only on $[\TT_0]$.
\end{itemize}
\end{assumption}

These assumptions are especially satisfied for pure $h$-refinement based on local bisection \cite{meshoneD} as well as for   the concrete strategy used in \cite{igafaermann} and \cite{resigabem}.
The latter strategy looks as follows: Let  $[\TT_\star]\in[\T]$. 
Let $\MM_\star\subseteq \NN_\star$ be a set of marked nodes.
To get the refined mesh $[\TT_{+}]:=\refine([\TT_\star],\MM_\star)$, we proceed as follows:
\begin{enumerate}[\rm(i)]
\item If both nodes of an element $T\in \TT_\star$ belong to $\MM_\star$, the element $T$ will be marked.
\item For all other nodes in $\MM_\star$, the multiplicity will be increased if it is less or equal to $p+1$, otherwise the elements which contain one of these nodes $z\in \MM_\star$, will be marked.
{\item Recursively, mark further elements $T'\in\TT_\star$ for refinement if there exists a marked element $T\in\TT_\star$ with $T\cap T'\neq\emptyset$ and $\check{h}_{\star,T'}>\check{\kappa}_0 \check{h}_{\star,T}$.
\item Refine all marked elements $T\in\TT_\star$ by bisection and hence obtain $[\TT_+]$.}
\end{enumerate}

{According to \cite{meshoneD}, the proposed recursion in step (iii) terminates and the generated partition $\TT_+$ guarantees (M1) with $\check{\kappa}_{\max}=2\check{\kappa}_0$.
The following proposition shows that also the assumptions (M2)--(M3) are satisfied.}

\begin{proposition}
The proposed refinement strategy $\refine(\cdot)$ used in \cite{igafaermann,resigabem} satisfies Assumption~\ref{ass:mesh assumption}, where  $\check{\kappa}_{\max}=2\check{\kappa}_0$ and $\Cmesh$  depends only on the initial partition of the parameter domain, i.e., $\TT_0$ transformed onto $[a,b]$.
\end{proposition}
\begin{proof}
For any partition $\TT_\star$ of $\Gamma$ and any subset of marked elements $\SS_\star\subseteq \TT_\star$, let $\widetilde{\refine}(\TT_\star,\SS_\star)$ be the partition obtained from the recursive bisection in step {\rm (iii)--(iv)} above.
{This local $h$-refinement procedure has been analyzed in \cite{meshoneD}.
According to \cite[Theorem~2.3]{meshoneD}, the recursion is well-defined and guarantees $\check{\kappa}_\star\le 2\check{\kappa}_0$ for all $\TT_\star\in \widetilde{\refine}(\TT_0)$.}

{To see {\rm(M2)}, \cite[Theorem 2.3]{meshoneD} guarantees the existence of some coarsest common refinement $\TT_\star \widetilde\oplus\TT_+\in\widetilde\refine(\TT_\star)\cap\widetilde\refine(\TT_+)$ such that}
\begin{align*}
|\TT_\star\widetilde{\oplus}\TT_{+}|\le |\TT_\star|+|\TT_{+}|-|\TT_0|.
\end{align*}
The corresponding nodes just satisfy $\NN_\star\oplus\NN_{+}=\NN_\star\cup\NN_{+}$.
There exists a finite sequence of meshes $\TT_\star=\widetilde\TT_1,\widetilde\TT_2=\widetilde\refine(\widetilde\TT_1,\SS_1),\dots,\widetilde\TT_\ell=\widetilde\refine(\widetilde\TT_{\ell-1},\SS_{\ell-1})=\TT_\star\widetilde{\oplus}\TT_{+}$ with suitable $\SS_j\subseteq \TT_j$ for $j=1,\dots,\ell-1$.
If we define $\MM_j\subseteq \NN_j$ as the set of all nodes in $\SS_j$, we see that the sequence $[\TT_\star]=[\TT_1],[\TT_2]=\refine([\TT_1],\MM_1),\dots[\TT_\ell]=\refine([\TT_{\ell-1},\MM_{\ell-1})$ satisfies $\TT_j=\widetilde\TT_j$ for $j=1,\dots\ell$.
By repetitively marking one single node, we obtain from $[\TT_\ell]$ a mesh $[\TT_\star\oplus \TT_+]$ with nodes $\NN_\star\oplus\NN_+=\NN_\star\cup\NN_+$ and $\# z=\max(\#_\star z,\#_+z)$, where $\#_\star$ resp. $\#_+$ denote the multiplicity in $\KK_\star$ resp. $\KK_+$ and, e.g.,  $\#_+ z:=0$ if $z\in\NN_\star\setminus\NN_+$.
There obviously holds
\begin{align*}
|\KK_\star\oplus\KK_+|=\sum_{z\in\NN_\star\cup\NN_+} \# z\le|\KK_\star|+|\KK_+|-|\KK_0|.
\end{align*}
Moreover, $[\TT_\star\oplus\TT_+]$ is clearly a refinement of $[\TT_+]$ as well.

Finally we consider {\rm(M3)}. As before we have  $\TT_1=\widetilde\refine(\TT_0,\SS_0),\dots,$$\TT_\ell=\widetilde\refine(\TT_{\ell-1},\SS_{\ell-1})$ for suitable $\SS_{j}\subseteq \TT_j$, $j=0,\dots,\ell-1$.
Note that there holds $|\SS_j|\le 2æ|\MM_j|$.
We denote $|\#_j|:=|\KK_{j+1}|-|\KK_{j}|-(|\NN_{j+1}|-|\NN_j|)$ as the number of multiplicity increases during the $j$-th refinement.
There holds 
{\begin{align*}
|\KK_{j+1}|-|\KK_j|= |\TT_{j+1}|-|\TT_j|+|\#_j|
\end{align*}
}and hence
\begin{align*}
|\KK_\ell|-|\KK_0|=|\TT_\ell|-|\TT_0|+\sum_{j=0}^{\ell-1} |\#_j|.
\end{align*}
The term $|\TT_\ell|-|\TT_0|$ can be estimated by $C\sum_{j=0}^{\ell-1} |\SS_j|$ with some constant $C>0$ which  depends only on the initial partition of the parameter domain, see \cite[Theorem~2.3]{meshoneD}, and hence by $2C\sum_{j=0}^{\ell-1}|\MM_j|$.
The estimate $|\#_j|\le |\MM_j|$ concludes the proof with $\Cmesh=2C+1$.
\end{proof}

\subsection{B-splines and NURBS}
\label{subsec:splines}
Throughout this subsection, we consider \textit{knots} $\check{\mathcal{K}}:=(t_i)_{i\in\Z}$ on $\R$ with multiplicity $\#t_i$ which satisfy
$t_{i-1}\leq t_{i}$ for $i\in \Z$ and $\lim_{i\to \pm\infty}t_i=\pm \infty$.
Let $\check{\mathcal{N}}:=\set{t_i}{i\in\Z}=\set{\check{{z}}_{j}}{j\in \Z}$ denote the corresponding set of nodes with $\check{{z}}_{j-1}<\check{{z}}_{j}$ for $j\in\Z$.
For $i\in\Z$, the $i$-th \textit{B-spline} of degree $p$ is defined inductively by
\begin{align}
\begin{split}
B_{i,0}&:=\chi_{[t_{i-1},t_{i})},\\
B_{i,p}&:=\beta_{i-1,p} B_{i,p-1}+(1-\beta_{i,p}) B_{i+1,p-1} \quad \text{for } p\in \N,
\end{split}
\end{align}
where, for $t\in \R$,
\begin{align*}
\beta_{i,p}(t):=
\begin{cases}
\frac{t-t_i}{t_{i+p}-t_i} \quad &\text{if } t_i\neq t_{i+p},\\
0 \quad &\text{if } t_i= t_{i+p}.
\end{cases}
\end{align*}
 We also use the notations $B_{i,p}^{\check{\mathcal{K}}}:=B_{i,p}$ and $\beta_{i,p}^{\check{\mathcal{K}}}:=\beta_{i,p}$ to stress the dependence on the knots~$\check{\mathcal{K}}$.
The following lemma collects some
basic properties of B-splines.

%

\begin{lemma}\label{lem:properties for B-splines}
Let $I=[a,b)$ be a finite interval and $p\in \N_0$.
Then, the following assertions  {\rm(i)--(vi)} hold:
\begin{enumerate}[\rm(i)]
\item \label{item:spline basis}
The set $\set{B_{i,p}|_I}{i\in \Z, B_{i,p}|_I\neq 0}$ is a basis for the space of all right-continuous $\check{\mathcal{N}}$-piecewise polynomials of degree lower or equal $p$ on $I$ which are, at each knot $t_i$, $p-\#t_i$ times continuously differentiable if $p-\#t_i\geq 0$.
\item \label{item:B-splines local} For $i\in\Z$,  $B_{i,p}$ vanishes outside the interval $[t_{i-1},t_{i+p})$. 
It is positive on the open interval $(t_{i-1},t_{i+p})$.
\item \label{item:B-splines determined} For $i\in \Z$,  $B_{i,p}$ is completely determined by the $p+2$ knots $t_{i-1},\dots,t_{i+p}$. 
\item\label{item:B-splines partition} The  B-splines of degree $p$ form a (locally finite) partition of unity, i.e.,
\begin{equation}
\sum_{i \in\Z} B_{i,p}=1\quad \text{on }\R.
\end{equation}
\end{enumerate}
\end{lemma}

\begin{proof}
The proof of~\eqref{item:spline basis} is found in {\cite[Theorem 6]{Boor-SplineBasics}}, and (\ref{item:B-splines local})--(\ref{item:B-splines determined}) are proved in \cite[Section 2]{Boor-SplineBasics}.
(\ref{item:B-splines partition}) is proved in \cite[page 9--10]{Boor-SplineBasics}.
\end{proof}
\begin{figure}[t] 
\psfrag{pacman (Section 5.3)}[c][c]{}
\psfrag{45}[r][r]{\tiny $45^\circ$}
\begin{center}
\includegraphics[width=1\textwidth]{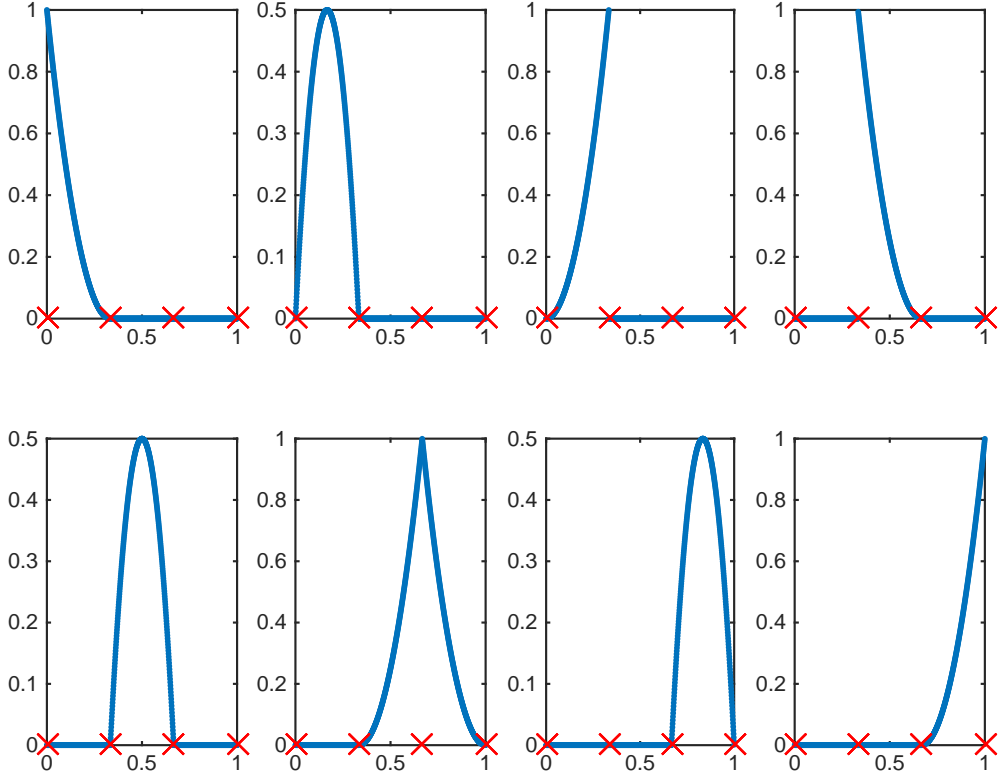}
\end{center}
\caption{B-splines on the interval $[0,1]$ corresponding to knot sequence $(\dots,0,0,0,1/3,1/3,1/3,2/3,2/3,1,1,1,\dots)$.} 
\label{fig:B-splines}
\end{figure}

In addition to the knots $\check{\mathcal{K}}=(t_i)_{i\in\Z}$, we consider positive weights $\mathcal{W}:=(w_i)_{i\in\Z}$ with $w_i>0$.
For $i\in \Z$ and $p\in \N_0$, we define the $i$-th 
NURBS by
\begin{equation}
R_{i,p}:=\frac{w_iB_{i,p}}{\sum_{\ell\in\Z}  w_{\ell}B_{\ell,p}}.
\end{equation}
We also use the notation $R_{i,p}^{\check{\mathcal{K}},\mathcal{W}}:=R_{i,p}$.
Note that the denominator is locally finite and positive.

For any $p\in\N_0$, we define the B-spline space
\begin{equation}
\mathscr{S}^p(\check{\mathcal{K}}):=\left\{\sum_{i\in\Z}a_i B_{i,p}:a_i\in\R\right\}
\end{equation}
as well as the NURBS space
\begin{equation}\label{eq:NURBS space defined} 
\mathscr{N}^p(\check{\mathcal{K}},\mathcal{W}):=\left\{\sum_{i\in\Z}a_i R_{i,p}:a_i\in\R\right\}=\frac{\mathscr{S}^p(\check{\mathcal{K}})}{\sum_{i\in \Z} w_\ell B_{\ell,p}^{\check{\mathcal{K}}}}.
\end{equation}

\subsection{Ansatz spaces}
\label{section:igabem}
Let $[\TT_0]$ be a given initial mesh with corresponding knots $\KK_0$ {such that $h_0 \le |\Gamma|/4$ for closed $\Gamma=\partial\Omega$.
We set $[\T]:=\refine([\TT_0])$.} Suppose that {$\mathcal{W}_0=(w_{i})_{i=1-p}^{N-p}$} are given initial weights 
with $N=|\KK_0|$ for closed $\Gamma=\partial\Omega$ resp.\ $N=|\KK_0|-(p+1)$ for open $\Gamma\subsetneqq\partial\Omega$.

If $\Gamma=\partial\Omega$ is closed, {we extend the transformed knot sequence $\check{\KK}_0=(t_i)_{i=1}^N$ arbitrarily to $(t_i)_{i\in \Z}$ with $t_{-p}=\dots=t_0=a$, $t_i\le t_{i+1}$, $\lim_{i\to \pm\infty}t_i=\pm \infty$ and $\mathcal{W}_0=(w_i)_{i\in \Z}$ with $w_i>0$.}
For the extended sequences, we also write  $\check{\mathcal{K}}_0$ and $\mathcal{W}_0$ and set
\begin{equation}\label{eq:X0:closed}
\XX_{0}:=\mathscr{N}^p(\check{\mathcal{K}}_{0},\mathcal{W}_{0})|_{[a,b)}\circ \gamma|_{[a,b)}^{-1}.
\end{equation}

If $\Gamma_{}\subsetneqq \partial\Omega$ is open,
we extend the sequences $\check{\KK}_0=(t_i)_{i=-p}^N$ and $\WW_0$ {arbitrarily to $(t_i)_{i\in \Z}$ with $t_i\le t_{i+1}$, $\lim_{i\to \pm\infty}t_i=\pm \infty$ and $\mathcal{W}_0=(w_i)_{i\in \Z}$ with $w_i>0$.} This allows to define
\begin{equation}\label{eq:X0:open}
\XX_0:=\mathscr{N}^p(\check{\mathcal{K}}_0,\mathcal{W}_0)|_{[a,b]}\circ \gamma^{-1}.
\end{equation}
Due to Lemma~\ref{lem:properties for B-splines}, 
 this definition does not depend on how the sequences are extended.

Let $[\TT_\star]\in[\T]$ be a mesh with knots $\KK_\star$. 
Via \emph{knot insertion} from $\KK_0$ to $\KK_\star$, one obtains unique corresponding weights $\WW_\star$.
These are chosen such that the denominators of the NURBS functions do not change.
In particular, this implies nestedness 
\begin{align}
\XX_\star\subseteq \XX_{+} \quad\text{for all }[\TT_\star]\in[\T], [\TT_{+}]\in\refine(\TT_\star),
\end{align}
 where the spaces $\XX_\star$ resp.\ $\XX_+$ are defined analogously to~\eqref{eq:X0:closed}--\eqref{eq:X0:open}.
Moreover, the weights are just convex combinations of $\WW_0$, wherefore
\begin{align}
w_{\min}:=\min(\WW_0)\le \min(\WW_\star)\le \max(\WW_\star)\le \max(\WW_0)=:w_{\max}.
\end{align}
For further details, we refer to, e.g., \cite[Section 4.2]{igafaermann}.


\section{Adaptive algorithm and main results}
\label{section:algorithm}

For each mesh $[\TT_\star]\in [\T]$, define the node-based error estimator 
\begin{subequations}\label{eq:weighted-residual}\begin{align}
\mu_{\star}^2=\sum_{z\in \NN_\star} \mu_{\star}(z)^2,
\end{align}
where the refinement indicators read
\begin{align}
\mu_{\star}(z)^2:= |\gamma^{-1}({\omega}_{\star}(z))| \norm{\partial_\Gamma(f-V\Phi_\star)}{L^2(\omega_{\star}(z))}^2
\quad\text{for all }z\in \NN_\star.
\end{align}\end{subequations}
Here, we must additionally suppose $f\in H^{1}(\Gamma)$ to ensure that $\mu_\star$ is well-defined.
It has been proved in~\cite{resigabem} that $\mu_\star$ is reliable, i.e., 
\begin{align}\label{eq:reliability}
\norm{\phi-\Phi_\star}{\H^{-1/2}(\Gamma)} \le \Crel\,\mu_\star,
\end{align}
 where $\Crel>0$ depends only on 
$p$, $w_{\rm min}$, $w_{\rm max}$, $\gamma$, and $\check{\kappa}_{\rm max}$.
{We note that the weighted-residual error estimator in the form $\mu_\star\simeq \norm{h_\star^{1/2}\partial_\Gamma(f-V\Phi_\star)}{L^2(\Gamma)}$ goes back to the works \cite{cs96, cc97}, where reliability \eqref{eq:reliability} is proved for standard 2D BEM with piecewise constants on polyhedral geometries, while the corresponding result for 3D BEM is found in \cite{cms}.}
We consider the following adaptive algorithm which employs the D\"orfler marking strategy~\eqref{eq:Doerfler} from \cite{dorfler} to single out nodes for refinement.

\begin{algorithm}\label{the algorithm}
\textbf{Input:} Adaptivity parameter $0<\theta<1$, $\Cmark\ge 1$, polynomial order $p\in \N_0$, initial mesh $[\TT_0]$, initial weights $\mathcal{W}_0$.\\
\textbf{Adaptive loop:} For each $\ell=0,1,2,\dots$ iterate the following steps {\rm(i)--(iv)}:
\begin{itemize}
\item[\rm(i)] Compute discrete approximation $\Phi_\ell\in\XX_\ell$ from Galerkin BEM.
\item[\rm(ii)] Compute refinement indicators $\mu_\ell({z})$
for all nodes ${z}\in\NN_\ell$.
\item[\rm(iii)] Determine an up to the multiplicative constant $\Cmark$  minimal set of nodes $\MM_\ell\subseteq\NN_\ell$ such that
\begin{align}\label{eq:Doerfler}
 \theta\,\mu_\ell^2 \le \sum_{{z}\in\MM_\ell}\mu_\ell({z})^2.
\end{align}
\item[\rm(iv)] Generate refined mesh $[\TT_{\ell+1}]:=\refine([\TT_\ell],\MM_\ell)$. 
\end{itemize}
\textbf{Output:} Approximate solutions $\Phi_\ell$ and error estimators $\mu_\ell$ for all $\ell \in \N_0$.
\end{algorithm}

Our main result is that the proposed algorithm is linearly convergent, even with the optimal algebraic rate.
For a precise statement of this assertion, let
$[\T_N]:=\set{[\TT_\star]\in[\T]}{|\KK_\star|-|\KK_0|\le N}$ be the finite set of all refinements having at most $N$ knots more than $[\TT_0]$.
Following~\cite{axioms}, we introduce an estimator-based approximation class $\A_s$ for $s>0$:
We write $\phi\in\A_s$ if 
\begin{align}
\norm{\phi}{\A_s}:=\sup_{N\in\N_0} \big( (N+1)^s \min_{[\TT_\star]\in[\T_N]} \mu_\star\big)<\infty.
\end{align}
In explicit terms, this just means that an algebraic convergence rate of $\mathcal{O}(N^{-s})$ for the estimator is possible, if the optimal meshes are chosen.
The following theorem is the main result of our work:

\def\Copt{C_{\rm opt}}
\begin{theorem}\label{thm:main}
Let $f\in H^1(\Gamma)$, so that the weighted-residual error estimator $\mu_\ell$ from~\eqref{eq:weighted-residual} is well-defined and that Algorithm~\ref{the algorithm} is driven by $\mu_\ell$.
We suppose that the Assumption~\ref{ass:mesh assumption} on the mesh-refinement holds true.
 Then, for each $0<\theta\le1$, there  exist constants $0<\qlin<1$ and $\Clin>0$ such that Algorithm~\ref{the algorithm} is linearly convergent in the sense of 
\begin{align}\label{eq:R-linear}
\mu_{\ell+n}\leq \Clin \,\qlin^n\,\mu_\ell \quad \text{for all }\ell,n\in\N_0.
\end{align}
In particular, this implies convergence
\begin{align}
 \Crel^{-1}\,\norm{\phi-\Phi_\ell}{\H^{-1/2}(\Gamma)} \le \mu_\ell
 \le \Clin\qlin^\ell\,\mu_0
 \xrightarrow{\ell\to\infty}0.
\end{align}%
Moreover, there is a constant $0<\theta_{\rm opt}<1$ such that for all $0<\theta<\theta_{\rm opt}$, there exists a constant $\Copt>0$ such that for all  $s>0$, it holds
\begin{align}\label{eq:optimal}
\phi\in\A_s\quad\Longleftrightarrow\quad \mu_\ell\le \frac{\Copt^{1+s}}{(1-\qlin^{1/s})^s} \norm{\phi}{\A_s}(|\KK_\ell|-|\KK_0|)^{-s}\quad \text{for all }\ell\in\N_0.
\end{align}
The constants $\qlin, \Clin$ depend only on $p, w_{\min},w_{\max}, \gamma, \theta$, and $\check{\kappa}_{\max}$ from {\rm(M1)}.
The constant $\theta_{\rm opt}$ depends only on $p, w_{\min},w_{\max}, \gamma$, and {\rm(M1)--(M3)}, and $\Copt$ depends additionally on $\theta$.
\end{theorem}

\begin{remark}
The proof of Theorem~\ref{thm:main} reveals that linear convergence~\eqref{eq:R-linear} only requires~{\rm(M1)}, while optimal rates~\eqref{eq:optimal} rely on~\eqref{eq:R-linear} and~{\rm(M2)--(M3)}.
\end{remark}%

The proof of Theorem~\ref{thm:main} is given in Section~\ref{section:rlinear}--\ref{section:optimal}.
The ideas essentially follow those of~\cite{axioms}, where an axiomatic approach of adaptivity for abstract problems is found.
We note, however, that \cite{axioms} only considers $h$-refinement, while the present formulation of  Algorithm~\ref{the algorithm} steers both, the $h$-refinement and the knot multiplicity increase.

If Algorithm~\ref{the algorithm} is steered by the Faermann estimator
\begin{subequations}\label{eq:Faermann}\begin{align}
\eta_{\star}^2=\sum_{z\in \NN_\star} \eta_{\star}(z)^2
\end{align}
with the refinement indicators
\begin{align}
\eta_{\star}(z)^2:=  |f-V\Phi_\star|_{H^{1/2}(\omega_{\star}(z))}^2
\quad\text{for all }z\in \NN_\star,
\end{align}\end{subequations}
instead of $\mu_\star$, we can prove at least plain convergence of the estimator to zero.
In contrast to the weighted-residual estimator which requires additional regularity $f\in H^1(\Gamma)$,  the Faermann estimator $\eta_\star$  
allows a right-hand side $f\in H^{1/2}(\Gamma)$. 
Moreover, $\eta_\star$ estimator is efficient and reliable
\begin{align}\label{eq:reliability2}
C_{\rm eff}^{-1} \eta_\star \le \norm{\phi-\Phi_\star}{\H^{-1/2}(\Gamma)} \le \Crel\,\eta_\star,
\end{align}
{where $C_{\rm eff}>0$ depends only on $\Gamma$, while $\Crel>0$ depends additionally on $p, \check{\kappa}_{\max}, w_{\min}, w_{\max}$ and $\gamma$;
see \cite[Theorem 3.1 and 4.4]{igafaermann}.
This equivalence of error and estimator puts some interest on the following convergence theorem which is, however, weaker than the statement of Theorem~\ref{thm:main}.}

\begin{theorem}\label{thm:faermann}
Let $f\in H^{1/2}(\Gamma)$.
We suppose that {\rm (M1)} from Assumption~\ref{ass:mesh assumption} for the mesh-refinement holds.
Then, for each $0<\theta\le1$, Algorithm~\ref{the algorithm} steered by the Faermann estimator \eqref{eq:Faermann} is convergent in the sense of
\begin{align}\label{eq:Faermann converges}
\eta_\ell \xrightarrow{\ell\to\infty}0.
\end{align}
According to \eqref{eq:reliability2}, this is equivalent to
\begin{align}
\norm{\phi-\Phi_\ell}{\H^{-1/2}(\Gamma)}
 \xrightarrow{\ell\to\infty}0.
\end{align}%
\end{theorem}
{\begin{remark}
The statements of Theorem~\ref{thm:main} and Theorem~\ref{thm:faermann} remain valid, if only adaptive $h$-refinement is used, i.e., if Algorithm~\ref{the algorithm} does not steer the knot multiplicity.
\end{remark}}

\section{Proof of Theorem~\ref{thm:main}, linear convergence~\eqref{eq:R-linear}}
\label{section:rlinear}

%
{As an auxiliary result, we need an inverse-type estimate for NURBS with respect to the fractional $\H^{-1/2}(\Gamma)$-norm}.
In the following, a result is stated and proved for the $\H^{-\sigma}(\Gamma)$-norm, where $0<\sigma<1$.
For piecewise polynomials, an analogous result is already found in \cite[Theorem 3.6]{inversesauter} resp. \cite[Theorem 3.9]{georgoulis}.
Our proof is inspired by \cite[Section 4.3]{inversefaermann}, where a similar result is found for piecewise constant functions as well as for piecewise affine and globally continuous functions in 1D.
For integer-order Sobolev norms, inverse estimates for NURBS are found in \cite[Section 4]{approximation}, and 
\eqref{eq:Cinv} is  proved in \cite[Theorem 3.1]{invest} for piecewise polynomials.

\begin{proposition}\label{prop:inverse estimate}
\label{prop:invest}
Let $[\TT_\star]\in [\T]$  and $0<\sigma<1$. 
Then, there is a constant $\Cinv>0$ such that
\begin{align}\label{eq:Cinvsig}
\norm{h_\star^{\sigma}\Psi_\star}{L^2(\Gamma)}\leq \Cinv\norm{\Psi_\star}{\H^{-\sigma}(\Gamma)}\quad  \text{for all } \Psi_\star\in\XX_\star.
\end{align}
For $\sigma=1/2$, it holds
\begin{align}\label{eq:Cinv}
\norm{h_\star^{1/2}\partial_\Gamma (V\Psi_\star)}{L^2(\Gamma)}+\norm{h_\star^{1/2}\Psi_\star}{L^2(\Gamma)}\leq \Cinv\norm{\Psi_\star}{\H^{-1/2}(\Gamma)}\quad  \text{for all } \Psi_\star\in\XX_\star.
\end{align}
The constant $\Cinv$ only depends on $\check{\kappa}_{\max}, p, w_{\min}, w_{\max}$, $\gamma$, and $\sigma$.
\end{proposition}

\begin{proof}
The proof is done in four steps.
First, we show that $\norm{h_\star^{\sigma}\psi}{L^2(\Gamma)}\lesssim \norm{\psi}{\H^{-\sigma}(\Gamma)}$ 
holds for all $\psi\in L^2(\Gamma)$ which satisfy
a certain assumption.
In the second step, we prove an auxiliary result for polynomials which is needed in the third one, where we show that all $\psi\in\XX_\star$ satisfy the mentioned assumption.
In the last step, we apply a recent result of~\cite{invest}, which then concludes the proof.

\noindent
{\bf Step 1:}
Let $\XX\subset L^2(\Gamma)$ satisfy the following assumption:
There exists a constant $q \in (0,1)$ such that for all $T\in\TT_\star$ and  all $\psi\in \XX$ there exists some connected subset $\Delta(T,\psi) \subseteq T$ of length $|\Delta(T,\psi)|\geq q|T|$ such that $\psi$ does not change its sign on $\Delta(T,\psi)$ and 
\begin{align}\label{eq:min-max}
\min_{x\in \Delta(T,\psi)}|\psi(x)|\geq q \,\max_{x\in T} |\psi(x)|.
\end{align} 
Then, there exists a constant $C>0$ which depends only on $q$ and $\check{\kappa}_\star$, such that 
\begin{align*}
\|h_\star^{\sigma} \psi \|_{L^2(\Gamma)} \leq C \|\psi\|_{\H^{-\sigma}(\Gamma)} \quad \text{for all }\psi\in \XX.
\end{align*}
For a compact nonempty interval $[c,d]= I \subseteq [a,b]$, we define the bubble function
\begin{align*}
P_I(t):=\begin{cases}
\left(\frac{t-c}{d-c}\cdot\frac{d-t}{d-c}\right)^2\quad &\text{if } t\in I, \\ 0 \quad &\text{if }t\in [a,b]\setminus I.
\end{cases}
\end{align*}  
It obviously satisfies $0\leq P_I\leq 1$ and $\supp P_I=I$.
A standard scaling argument proves
\begin{align}\label{eq:scaling1}
C_1|I|\leq \|P_I\|_{L^2(I)}^2 \le \|P_I\|_{L^1(I)} \leq C_2 |I|
\end{align}
and
\begin{align}\label{eq:scaling2}
|I|^2\|P_I'\|_{L^2(I)}^2\leq C_3\|P_I\|_{L^2(I)}^2
\end{align}
with generic constants $C_1,C_2,C_3>0$ which do not depend on $I$.
For each $T\in \TT_\star$, let $I(T,\psi)$ be some interval with $\gamma(I(T,\psi))=\Delta(T,\psi)$.
With the arclength parametrization $\gamma_L$, we define, for all $T\in \TT_\star$,  the functions $P_{\Delta(T,\psi)}:=P_{I(T,\psi)}\circ\gamma_L$  and the coefficients
\begin{align}\label{eq:cT}
c_T:=\mathrm{sgn}(\psi|_{\Delta(T,\psi)})h_{\star,T}^{2\sigma}  \min_{x\in \Delta(T,\psi)} |\psi(x)|.
\end{align}
Note that~\eqref{eq:scaling1}--\eqref{eq:scaling2} hold for $P_{\Delta(T,\psi)}$ with $I$ simply replaced by $\Delta(T,\psi)$ and with $(\cdot)'$ replaced by the arclength derivative $\partial_\Gamma$.
By definiti gon of the dual norm, it holds 
\begin{align}\label{eq:test fct}
\|\psi\|_{\H^{-\sigma}(\Gamma)}\geq \frac{|\dual{\psi}{\chi}|}{\|\chi\|_{H^{\sigma}(\Gamma)}}
\quad\text{with, e.g., }
\chi:= \sum_{T\in\TT_\star} c_T P_{\Delta(T,\psi)}\in H^1(\Gamma)\subset H^{\sigma}(\Gamma).
\end{align}
First, we estimate the numerator in \eqref{eq:test fct}:
\begin{eqnarray*}
|\dual{\psi}{\chi}|
&=& \left|\sum_{T\in\TT_\star} \int_{T}  \psi(x) c_T P_{\Delta(T,\psi)}(x) \,dx\right|\\
&\stackrel{\eqref{eq:cT}}{=}& \sum_{T\in\TT_\star} h_{\star,T}^{2\sigma} \min_{x\in \Delta(T,\psi)}|\psi(x)|^2\,\|P_{\Delta(T,\psi)}\|_{L^1(\Delta(T,\psi))}\\
&\stackrel{\eqref{eq:min-max}}{\geq}& {q^2} \sum_{T\in\TT_\star}  h_{\star,T}^{2\sigma}\max_{x\in T} |\psi(x)|^2 {\|P_{\Delta(T,\psi)}\|_{L^1(\Delta(T,\psi))}}\\
&\stackrel{\eqref{eq:scaling1}}{\geq}& {C_1q^3 \sum_{T\in\TT_\star} h_{\star,T}^{2\sigma}\,\|\psi\|_{L^2(T)}^2 }\\
&=& C_1 q^3  \|h_\star^{\sigma}\psi\|_{L^2(\Gamma)}^2.
\end{eqnarray*}
It remains to estimate the denominator in \eqref{eq:test fct}:
We first note that it holds $|u|_{H^{\sigma}(I)}^2\lesssim |I|^{1-\sigma}\norm{u'}{L^2(I)}$ for any  interval $I\subset \R$ of finite length and $u\in H^1(I)$.
This is already stated in \cite[Lemma 7.4]{carsfaer}. 
However, a detailed proof is  given only for $1/2<\sigma<1$.
For $0<\sigma\le 1/2$ this inequality can be shown exactly as in the proof of~\cite[Lemma 4.5]{resigabem}, where only $\sigma=1/2$ is considered.
This, together with \eqref{eq:equivalent Hsnorm}, implies for any connected $\omega\subseteq \Gamma$ with  $|\omega|\le \frac{3}{4}|\Gamma|$ that 
\begin{align}\label{eq:carsfaer}
|u|_{H^{\sigma}(\omega)}\lesssim |\omega|^{1-\sigma} \norm{\partial_\Gamma u}{L^2(\omega)}\quad\text{for all } u\in H^1(\Gamma).
\end{align} 
The hidden constant in \eqref{eq:carsfaer} depends only on $\sigma$ and $\Gamma$. 
\eqref{eq:carsfaer} is applicable for any node patch $\omega_\star(z)$ since we assumed in Section~\ref{section:igabem} that $h_0\le |\Gamma|/4$ if $\Gamma=\partial\Omega$
With \cite[Lemma 2.3]{faermann2d},  we hence see 
\begin{eqnarray*}
|\chi|_{H^{\sigma}(\Gamma)}^2&\stackrel{\mbox{\scriptsize\cite{faermann2d}}}\lesssim &
\|h_{\star}^{-\sigma}\chi\|_{L^2(\Gamma)}^2+\sum_{z\in\mathcal{N}_\star} |\chi|_{H^{\sigma}(\omega_z)}^2 \\
&{\stackrel{\mbox{\scriptsize\eqref{eq:carsfaer}}}\lesssim}& 
\|h_{\star}^{-\sigma}\chi\|_{L^2(\Gamma)}^2+\sum_{z\in\mathcal{N}_\star} \norm{h_\star^{1-\sigma} \partial_\Gamma\chi}{L^2(\omega_z)}^2  
\\&\simeq& \|h_{\star}^{-\sigma}\chi\|_{L^2(\Gamma)}^2+\sum_{T\in\TT_\star}  \norm{h_\star^{1-\sigma} \partial_\Gamma\chi}{L^2(\Delta(T,\psi))}^2  \\
&=&\norm{h_\star^{-\sigma}\chi}{L^2(\Gamma)}^2 + \sum_{T\in\TT_\star}  h_{\star,T}^{2-2\sigma}c_T^2 \norm{\partial_\Gamma P_{\Delta(T,\psi)}}{L^2(\Delta(T,\psi))}^2\\
&\stackrel{\eqref{eq:scaling2}}\le& {\norm{h_\star^{-\sigma}\chi}{L^2(\Gamma)}^2}+C_3\sum_{T\in\TT_\star}  h_{\star,T}^{2-2\sigma}c_T^2|\Delta(T,\psi)|^{-2}  \norm{ P_{\Delta(T,\psi)}}{L^2(\Delta(T,\psi))}^2\\
&\simeq& \|h_\star^{-\sigma}\chi\|_{L^2(\Gamma)}^2.
\end{eqnarray*}
This yields
\begin{align*}
\|\chi\|_{H^{\sigma}(\Gamma)}^2
= {\|\chi\|_{L^2(\Gamma)}^2 + |\chi|_{H^{\sigma}(\Gamma)}^2}
\lesssim \|h_\star^{-\sigma} \chi\|_{L^2(\Gamma)}^2,
\end{align*}
{where the hidden constant depends only on $\check{\kappa}_{\max}, \sigma$, and $\gamma$}.
With
\begin{eqnarray*}
\|h_\star^{-\sigma} \chi\|_{L^2(\Gamma)}^2&=& \sum_{T\in \TT_\star} h_{\star,T}^{-2\sigma} c_T ^2\| P_{\Delta(T,\psi)}\|_{L^2(\Delta(T,\psi))}^2
\stackrel{\eqref{eq:scaling1}}{\le} C_2\sum_{T\in\TT_\star}h_{\star,T}^{-2\sigma} c_T^2 |\Delta(T,\psi)| 
\\&\stackrel{\eqref{eq:cT}}{=}& C_2\sum_{T\in\TT_\star} h_{\star,T}^{2\sigma}\min_{x\in \Delta(T,\psi)} |\psi(x)|^2|\Delta(T,\psi)| \\
&\leq &{C_2 \sum_{T\in \TT_\star} h_{\star,T}^{2\sigma} \norm{\psi}{L^2(\Delta(T,\psi))}^2}\le C_2 \|h_\star^{\sigma} \psi \|_{L^2(\Gamma)}^2,
\end{eqnarray*}
we finish the first step.

\noindent
{\bf Step 2:}
For some fixed polynomial degree $p\in\N_0$, there exists a constant $q_1\in(0,1)$ such that for all polynomials $F$ of degree $p$ on $[0,1]$ there exists some interval $I\subseteq [0,1]$ of length $|I|\geq q_1$ with 
\begin{align}
\min_{t\in I}|F(t)|\geq q_1 \,\max_{t\in [0,1]} |F(t)|.
\end{align}
Instead of considering general polynomials $\mathcal{P}^p([0,1])$ of degree $p$, it is sufficient to consider the following subset
\begin{align*}
\mathcal{M}:=\set{F\in \mathcal{P}^p([0,1])}{\|F\|_{\infty}=1}.
\end{align*}
Note that $\mathcal{M}$ is a compact subset of $L^{\infty}([0,1])$ and that differentiation $(\cdot)'$ is a continuous mapping on $\mathcal{M}$ due to finite dimension.
This especially implies boundedness $\sup_{F\in \mathcal{M}} \norm{F'}{\infty}\leq C_4<\infty$.
We may assume $C_4>2$.
For given $F\in \mathcal{M}$, we define the interval $I$ as follows:
Without loss of generality, we assume that the maximum of $|F|$ is attained  at some $t_1\in[0,1/2]$ and that $F(t_1)=1$.
We set $t_3:=t_1+C_4^{-1}\in (t_1,1]$ and $t_2:=t_1+C_4^{-1}/2\in (t_1,3/4]$ and $I:=[t_1,t_2]$.
Then, $|I|=1/(2C_4)$ and for all $t\in I$ it holds
\begin{align*}
1/2\le  C_4(t_3-t)=F(t_1)+C_4(t_1-t)\leq F(t_1)+\norm{F'}{\infty}(t_1-t)\le F(t)=|F(t)|.
\end{align*}
Altogether, we thus have
\begin{align*}
q_1:= 1/(2C_4)\leq 1/2 \le \min_{t\in I} |F(t)| \quad \text{and} \quad |I|=q_1
\end{align*}
and conclude this step.

\noindent
{\bf Step 3:}
We show that $\XX_\star$ satisfies the assumption of Step 1 and hence conclude $\norm{h_\star^{\sigma}\Psi_\star}{L^2(\Gamma)}$ $\lesssim \norm{\Psi_\star}{\H^{-\sigma}(\Gamma)}$ for all $\Psi_\star\in\XX_\star$:
Let $\check{T}\subset [a,b]$ be the interval with $\gamma(\check{T})=T$ and  $\check{\psi}:=\psi\circ\gamma|_{\check{T}}$.
Since $|I|\simeq |\gamma(I)|$ for any interval $I\subseteq [a,b]$, where the hidden constants  depend only on $\gamma$, we just have to find a uniform constant $q_2\in (0,1)$ and some interval $I\subseteq \check{T}$ of length $|I|\geq q_2 |\check{T}|$ with 
\begin{align}\label{eq:parameter NURBS ass}
\min_{t\in I}|\check{\psi}(t)|\geq q_2 \max_{x\in \check{T}} |\check{\psi}(t)|.
\end{align} 
The function $\check{\psi}$ has the form $F/w$ with a polynomial $F$ of degree $p$ and the weight function $w$, which is also a polynomial of degree $p$ and which satisfies
$w_{\min}\leq w\leq w_{\max}$.
Hence, \eqref{eq:parameter NURBS ass} is especially satisfied if
\begin{align}\label{eq:parameter spline ass}
\min_{t\in I}|F(t)|\geq q_1\,\frac{w_{\max}}{w_{\min}} \, \max_{x\in \check{T}} |F(t)|.
\end{align}
After scaling to the interval $[0,1]$, we can apply Step 2 and conclude this step.
Altogether, this  proves \eqref{eq:Cinvsig}.

\noindent
{\bf Step 4:}
According to~\cite{invest}, it holds $\norm{h_\star^{1/2}\partial_\Gamma(V\psi)}{L^2(\Gamma)}\lesssim \norm{h_\star^{1/2}\psi}{L^2(\Gamma)}+\norm{\psi}{\H^{-1/2}(\Gamma)}$ for all $\psi \in L^2(\Gamma)$, where the hidden constant depends only on $\Gamma$, $\gamma$, and $\check{\kappa}_{\max}$.
Together with Step 3, this shows \eqref{eq:Cinv}.
\end{proof}

The proof of linear convergence~\eqref{eq:R-linear} will be done with the help of some auxiliary (and purely theoretical) error estimator $\widetilde\rho_\star$. The latter relies on the following definition of an equivalent mesh-size function which respects the multiplicity of the knots.

\begin{proposition}\label{lem:mesh function}
Assumption~\ref{ass:mesh assumption} {\rm (M1)} implies the existence of a modified mesh-size function $\widetilde{h}:[\T]\to L^\infty(\Gamma)$ with the following properties:
There exists a constant $\Chup>0$ and $0<\qctr<1$ which depend only on $\check{\kappa}_{\max}, p$ and $\gamma$ such that for all $[\TT_\star]\in [\T]$ and all refinements $[\TT_{+}]\in\refine([\TT_\star])$, the corresponding mesh-sizes $\widetilde{h}_\star:=\widetilde{h}([\TT_\star])$ and $\widetilde{h}_{+}:=\widetilde{h}([\TT_{+}])$ satisfy equivalence
\begin{align}
\Chdown \check{h}_\star\le \widetilde{h}_\star\le \Chup \check{h}_\star,
\end{align}
reduction
\begin{align}
\widetilde{h}_{+}\le \widetilde{h}_\star,
\end{align}
as well as contraction on the patch of refined elements
\begin{align}
\widetilde{h}_{+}|_{\omega_{+}([\TT_+]\setminus [\TT_\star])}\le \qctr \widetilde{h}_\star|_{\omega_{+}([\TT_+]\setminus [\TT_\star])}.
\end{align}
Note that   $\omega_{+}( [\TT_{+}]\setminus[\TT_{\star}])=\omega_{{\star}}( [\TT_{\star}]\setminus[\TT_{+}])$, which follows from $\bigcup([\TT_+]\setminus [\TT_\star])=\bigcup ([\TT_\star]\setminus[\TT_+])$ and the fact that the application of $\omega_{+}$ resp. $\omega_{\star}$ only adds elements of $\TT_{\star}\cap\TT_{+}$.
\end{proposition}

\begin{proof}
For all $[\TT_\star]\in\T$, we define $\widetilde{h}_\star\in L^\infty(\Gamma)$ by 
\begin{align*}
\widetilde{h}_\star|_T=|\gamma^{-1}(\omega_\star(T))| \cdot q_1^{\sum_{z\in\NN_\star\cap\omega_\star(T)}\#z} \quad\text{for all }T\in\TT_\star,
\end{align*}
where $0<q_1<1$ is fixed later. Clearly, $\widetilde h_\star\simeq \check{h}_\star$, where the hidden equivalence constants depend only on $\check{\kappa}_\star$, $p$, and $q_1$.
Let $x\in\Gamma$. First, suppose $x\not\in\omega_{+}([\TT_{+}]\setminus[\TT_\star])\cup\NN_{+}$, i.e., neither
the element $[T]\in[\TT_{+}]$ containing $x$ nor its neighbors result from  $h$-refinement or from multiplicity increase.
Then, $\widetilde{h}_{+}(x)=\widetilde{h}_\star(x)$.
Second, suppose $x\in \omega_{+}([\TT_{+}]\setminus[\TT_\star])\setminus\NN_{+}$, i.e., the element $[T']\in [\TT_{+}]$ containing $x$ or one of its neighbors result from $h$-refinement and/or multiplicity increase.
If only multiplicity increase took place, we get 
\begin{align*}
q_1^{\sum_{z\in\NN_+\cap\omega_+(T')} \# z}\leq q_1\cdot q_1^{\sum_{z\in\NN_\star\cap\omega_\star(T)} \# z}.
\end{align*}
In the other case, consider the father $[T]\in[\TT_\star]$ of $[T']$, i.e., $T'\subseteq T$.
Note that
\begin{align*}
|\gamma^{-1}(\omega_{+}(T'))|\le q_2\, |\gamma^{-1}(\omega_\star(T))|
\end{align*}
with a constant $0<q_2<1$  which depends only on $\check{\kappa}_{\max}$. 
Choose $0<q_1<1$ sufficiently large such that 
\begin{align*}
{q_2}/{q_1^{4p}}<1.
\end{align*} 
This choice yields $\widetilde{h}_{+}(x)\le ({q_2}/q_1^{4p}) \cdot\widetilde{h}_\star (x)$, since $\NN_\star\cap\omega_\star(T)$ contains at most $4$ nodes.
Therefore, we conclude the proof with $\qctr:=\max(q_1,{q_2}/q_1^{4p})$.
\end{proof}

\begin{remark}
Note that the construction of $\widetilde h_\star$ in Proposition~\ref{lem:mesh function} even ensures contraction
$\widetilde{h}_{+}|_{\omega_{+}(T)}\le \qctr \widetilde{h}_\star|_{\omega_{+}(T)}$
if ${[T]\in[\TT_{+}]\setminus [\TT_\star]}$ is obtained by $h$-refinement, while the multiplicity of all nodes $z\in\NN_+\cap\omega_+(T)$ is arbitrarily chosen $\#z\in\{1,\dots,p+1\}$.
In explicit terms, this allows for instance to set the multiplicity of all nodes $z\in\NN_+\cap\omega_+(T)$ to $\#z:=1$, if $T$ is obtained by $h$-refinement.\qed
\end{remark}

For any $[\TT_\star]\in[\T]$, we define the auxiliary estimator
\begin{align}\label{eq:tilde rho}
\widetilde{\rho}^{\,2}_\star:=\sum_{T\in\TT} \widetilde{\rho}^{\,2}_{\star}(T)
\quad \text{with}\quad
\widetilde{\rho}^{\,2}_{\star}(T):=\norm{\widetilde{h}_\star^{1/2}\partial_\Gamma (f-V\Phi_\star)}{L^2(T)}^2
\end{align}
which employs the novel mesh-size function $\widetilde h_\star$ from Proposition~\ref{lem:mesh function}.
{Obviously the estimators $\mu_\star$ and $\widetilde{\rho}_\star$ are locally equivalent
\begin{align}\label{eq:loceqtilde}
{\widetilde{\rho}_\star^{\,2}(T) \lesssim \mu_\star^2(z)\lesssim\sum_{T'\in\TT_\star\atop z\in T'} \widetilde\rho_\star^{\,2}(T')
\quad\text{for all }z\in \NN_\star\text{ and }T\in\TT_\star \text{ with }z\in T,}
\end{align}
where the hidden constants depend only on $\check{\kappa}_{\max}$, $p$, and $\gamma$.
}
The proof of the following lemma is inspired by \cite[Proposition 3.2]{fkmp} resp. \cite[Lemma 8.8]{axioms}, where only $h$-refinement is considered.

\begin{lemma}[estimator reduction of $\widetilde{\rho}$] 
Algorithm~\ref{the algorithm} guarantees  
\begin{align}\label{eq:estimator reduction}
\widetilde{\rho}_{\ell+1}^{\,2} \leq \qest\widetilde{\rho}_\ell^{\,2} +\Cest \norm{\Phi_{\ell+1}-\Phi_\ell}{\H^{-1/2}(\Gamma)}^2\quad \text{for all }\ell\geq 0.
\end{align}
The  constants $0<\qest<1$ and $\Cest>0$ depend only on $\check{\kappa}_{\max},p,w_{\min},w_{\max},\gamma,$ and $\theta$. 
\end{lemma}
\begin{proof}
The proof is done in several steps.

\noindent
{\bf Step 1:} With the inverse estimate \eqref{eq:Cinv},  there holds the following stability property for any measurable $\Gamma_0\subseteq \Gamma$
\begin{align*}
\begin{split}
&\left|\norm{ \widetilde{h}_{\ell+1}^{1/2} \partial_\Gamma (f-V\Phi_{\ell+1})}{L^2(\Gamma_0)}-\norm{ \widetilde{h}_{\ell+1}^{1/2} \partial_\Gamma (f-V\Phi_{\ell})}{L^2(\Gamma_0)}\right|
\\ &\quad\leq \norm{ \widetilde{h}_{\ell+1}^{1/2} \partial_\Gamma V(\Phi_{\ell+1}-\Phi_{\ell})}{L^2(\Gamma_0)}\leq C \norm{\Phi_{\ell+1}-\Phi_\ell}{\H^{-1/2}(\Gamma)},
\end{split}
\end{align*}
with a constant $C>0$ which depends only on $\Chup,\CinvV$, and $\gamma$.

\noindent
{\bf Step 2:}
With Proposition~\ref{lem:mesh function}, we split the estimator into a contracting and into a non-contracting part
\begin{align*}
\widetilde{\rho}_{\ell+1}^{\,2}=\norm{ \widetilde{h}_{\ell+1}^{1/2} \partial_\Gamma(f-V\Phi_{\ell+1})}{L^2(\omega_{\ell+1}([\TT_{\ell+1}]\setminus[\TT_{\ell}]))}^2+\norm{\widetilde{h}^{1/2}_{\ell+1} \partial_\Gamma(f-V\Phi_{\ell+1})}{L^2(\Gamma\setminus \omega_{\ell+1}([\TT_{\ell+1}]\setminus[\TT_{\ell}]))}
\end{align*}
Step 1, the Young inequality, and Proposition~\ref{lem:mesh function} show, for arbitrary $\delta>0$, that
\begin{align*}
&\norm{ \widetilde{h}_{\ell+1}^{1/2}\partial_\Gamma (f-V\Phi_{\ell+1})}{L^2(\omega_{{\ell+1}}([\TT_{\ell+1}]\setminus[\TT_{\ell}]))}^2 \\
&\quad \leq (1+\delta) \norm{ \widetilde{h}_{\ell+1}^{1/2} \partial_\Gamma (f-V\Phi_{\ell})}{L^2(\omega_{{\ell+1}}([\TT_{\ell+1}]\setminus[\TT_{\ell}]))}^2+(1+\delta^{-1}) C^2 \norm{\Phi_{\ell+1}-\Phi_\ell}{\H^{-1/2}(\Gamma)}^2\\
&\quad\leq (1+\delta) \qctr\norm{ \widetilde{h}_{\ell}^{1/2} \partial_\Gamma (f-V\Phi_\ell)}{L^2(\omega_{{\ell}}([\TT_{\ell}]\setminus[\TT_{\ell+1}]))}^2 +(1+\delta^{-1}) C^2 \norm{\Phi_{\ell+1}-\Phi_\ell}{\H^{-1/2}(\Gamma)}^2.
\end{align*}
Analogously, we get
\begin{align*}
&\norm{ \widetilde{h}_{\ell+1}^{1/2}\partial_\Gamma (f-V\Phi_{\ell+1})}{L^2(\Gamma\setminus\omega_{{\ell+1}}([\TT_{\ell+1}]\setminus[\TT_{\ell}]))}^2 \\
&\quad\leq (1+\delta) \norm{ \widetilde{h}_{\ell}^{1/2} \partial_\Gamma(f-V\Phi_\ell)}{L^2(\Gamma\setminus \omega_{{\ell}}([\TT_{\ell}]\setminus[\TT_{\ell+1}]))}^2 +(1+\delta^{-1}) C^2 \norm{\Phi_{\ell+1}-\Phi_\ell}{\H^{-1/2}(\Gamma)}^2.
\end{align*}
Combining these estimates, we end up with
\begin{align}\label{eq:final rho lp1}
\begin{split}
\widetilde{\rho}_{\ell+1}^{\,2}&\leq (1+\delta)\widetilde{\rho}_\ell^{\,2}-(1+\delta)(1- \qctr)\norm{ \widetilde{h}_{\ell}^{1/2} \partial_\Gamma(f-V\Phi_\ell)}{L^2(\omega_{{\ell}}([\TT_{\ell}]\setminus[\TT_{\ell+1}]))}^2 \\&\quad +2(1+\delta^{-1}) C^2 \norm{\Phi_{\ell+1}-\Phi_\ell}{\H^{-1/2}(\Gamma)}^2.
\end{split}
\end{align}
{\bf Step 3:} { Local equivalence \eqref{eq:loceqtilde} and the D\"orfler marking \eqref{eq:Doerfler} for $\mu_\ell$ imply}
\begin{align*}
\theta \widetilde{\rho}_\ell^{\,2}\simeq \theta \mu_\ell^2 \leq \sum_{z\in \MM_\ell}\mu_\ell(z)^2\simeq \sum_{\substack{T\in \TT_\ell\\T\subseteq\omega_{{\ell}}(\MM_\ell)}} \widetilde{\rho}_\ell(T)^2,
\end{align*}
where the hidden constants depend only on $\check{\kappa}_{\rm max}$, $p$, and $\gamma$.
Hence, $\widetilde{\rho}_\ell$ satisfies some D\"orfler marking  with a certain parameter $0<\widetilde{\theta}<1$.
With  $\MM_\ell\subseteq \bigcup([\TT_\ell]\setminus [\TT_{\ell+1}])$, \eqref{eq:final rho lp1} hence becomes 
\begin{align*}
\widetilde{\rho}_{\ell+1}^{\,2}&\leq \left((1+\delta)-(1+\delta)(1- \qctr)\widetilde\theta\right)\widetilde\rho_\ell^{\,2} +2(1+\delta^{-1}) C^2 \norm{\Phi_{\ell+1}-\Phi_\ell}{\H^{-1/2}(\Gamma)}^2.
\end{align*}
Choosing $\delta$ sufficiently small, we prove \eqref{eq:estimator reduction} with $\Cest:=2(1+\delta^{-1})C^2$ and 
$\qest:=(1+\delta)\big(1-(1-\qctr)\widetilde{\theta}\big)<1$.
\end{proof}

\begin{proof}[Proof of linear convergence \eqref{eq:R-linear}]
Due to the properties of the weakly-singular integral operator $V$, the bilinear form $A(\phi,\psi):=\dual{V\phi}{\psi}_{\Gamma}$ defines even a scalar product, and the induced norm $\norm{\psi}V:=A(\psi,\psi)^{1/2}$ is an equivalent norm on $\H^{-1/2}(\Gamma)$. According to nestedness of the ansatz spaces $\XX_{\ell}\subset\XX_{\ell+1}$, the Galerkin orthogonality implies the Pythagoras theorem
\begin{align*}
 \norm{\phi-\Phi_{\ell+1}}{V}^2 + \norm{\Phi_{\ell+1}-\Phi_\ell}{V}^2 = \norm{\phi-\Phi_\ell}{V}^2
 \quad\text{for all }\ell\in\N_0.
\end{align*}
Together with the estimator reduction~\eqref{eq:estimator reduction} and reliability \eqref{eq:reliability} 
\begin{align*}
 \norm{\phi-\Phi_\ell}{V} \simeq \norm{\phi-\Phi_\ell}{\H^{-1/2}(\Gamma)} 
 \lesssim \mu_\ell \simeq \widetilde\rho_\ell,
\end{align*}
this implies the existence of $0<\kappa,\lambda<1$, {which depend only on $\Crel,\Cest$ and $\qest$}, such that $\Delta_\star:=\norm{\phi-\Phi_\star}{V}^2+\lambda\,\widetilde\rho_\star^{\,2}\simeq\widetilde\rho_\star^{\,2} $
satisfies
\begin{align*}
 \Delta_{\ell+1} \le \kappa\,\Delta_\ell
 \quad\text{for all }\ell\in\N_0;
\end{align*}
see, e.g., \cite[Theorem~4.1]{fkmp}, while the original idea goes back to~\cite{ckns}. From this, we infer 
\begin{align*}
 \mu_{\ell+n}^2\simeq\,\widetilde\rho_{\ell+n}^{\,2} \simeq \Delta_{\ell+n} \le \kappa^n\,\Delta_\ell
 \simeq\kappa^n\,\widetilde\rho_{\ell}^{\,2} \simeq \kappa^n\,\mu_{\ell}^2
 \quad\text{for all }\ell,n\in\N_0
\end{align*}
and hence conclude the proof.
\end{proof}%

\section{Proof of Theorem~\ref{thm:main}, optimal convergence~\eqref{eq:optimal}}
\label{section:optimal}



As in the previous section, we define an auxiliary error estimator.
For each $[\TT_\star]\in[\T]$, let
\begin{align}\label{eq:rho}
{\rho}^2_\star:=\sum_{T\in\TT} {\rho}_{\star}(T)^2
\quad\text{with}\quad{\rho}_{\star}(T)^2:=\norm{\check{h}_\star^{1/2}\partial_\Gamma (f-V\Phi_\star)}{L^2(T)}^2.
\end{align}
{Note that the estimators $\mu_\star$ and $\rho_\star$ are again locally equivalent
\begin{align}\label{eq:loceq}
{\rho_\star^{\,2}(T) \le \mu_\star^{\,2}(z)\lesssim\sum_{T'\in\TT_\star\atop z\in T'}\rho_\star^{\,2}(T')\quad\text{for all }z\in \NN_\star \text{ and } T\in\TT_\star \text{ with } z\in T,}
\end{align}
where the hidden constant depends only on $\check{\kappa}_{\max}$.
}
{Analogous versions of the next two lemmas are already proved in \cite[Proposition 4.2 and 4.3]{fkmp} for $h$-refinement and piecewise constants; see also \cite[Propostion~5.7]{axioms} for discontinuous piecewise polynomials  and $h$-refinement.
The proof for Lemma~\ref{lem:stability rho} is essentially based on Proposition~\ref{prop:inverse estimate}.
The proof of Lemma~\ref{lem:discrete reliability}  requires 
the construction of a Scott-Zhang type operator \eqref{eq:scotty} which  is not necessary in \cite{fkmp,axioms}, since both works consider discontinuous piecewise polynomials.
}
\begin{lemma}[stability of $\rho$]\label{lem:stability rho}
Let $[\TT_\star]\in [\T]$ and $[\TT_+]\in \refine(\TT_\star)$.
For $\SS\subseteq \TT_\star\cap \TT_+$ there holds
\begin{align}\label{eq:stability}
\Big|\Big(\sum_{T\in\SS} \rho_{+}(T)^2\Big)^{1/2}-\Big(\sum_{T\in\SS} \rho_\star(T)^2\Big)^{1/2}\Big| \le \Cstab \norm{\Phi_{+}-\Phi_\star}{\H^{-1/2}(\Gamma)},
\end{align}
where $\Cstab>0$ depends only on the parametrization $\gamma$ and the constant $\CinvV$ of Proposition~\ref{prop:inverse estimate} 
\end{lemma}
\begin{proof}
For all subsets $\Gamma_0\subseteq \Gamma$, it holds
\begin{align}\label{eq:prestability}
\notag&\left|\norm{\check{h}_{+}^{1/2}\partial_\Gamma (f-V\Phi_{+})}{L^2(\Gamma_0)} -\norm{\check{h}_{+}^{1/2}\partial_\Gamma (f-V\Phi_\star)}{L^2(\Gamma_0)}\right| \le \norm{\check{h}_{+}^{1/2}\partial_\Gamma V(\Phi_{+}-\Phi_\star)}{L^2(\Gamma_0)} \\
&\quad\lesssim \norm{{h}_{+}^{1/2}\partial_\Gamma V(\Phi_{+}-\Phi_\star)}{L^2(\Gamma_0)} \le \CinvV \norm{\Phi_{+}-\Phi_\star}{\H^{-1/2}(\Gamma)}.
\end{align}
The choice $\Gamma_0=\bigcup\SS$ shows stability
\begin{align*}
\Big|\Big(\sum_{T\in\SS} \rho_{+}(T)^2\Big)^{1/2}-\Big(\sum_{T\in\SS} \rho_\star(T)^2\Big)^{1/2}\Big| \le \CinvV \norm{\Phi_{+}-\Phi_\star}{\H^{-1/2}(\Gamma)},
\end{align*}
and we conclude the proof.
\end{proof}

\begin{lemma}[discrete reliability of $\rho$]\label{lem:discrete reliability} 
There exist constants $\Cdrel, \Cref>0$, which depend only on $\check{\kappa}_{\max}, p, w_{\min}, w_{\max},$ and $\gamma$, such that for all refinements $[\TT_{+}]\in \refine([\TT_\star])$ of $[\TT_\star]\in [\T]$ there exists a subset $\RR_\star(\TT_+)\subseteq \TT_\star$ with
\begin{align}
\norm{\Phi_{+}-\Phi_\star}{\H^{-1/2}(\Gamma)}^2\le \Cdrel\sum_{T\in \RR_\star(\TT_+)} \rho_\star(T)^2
\end{align}
as well as
\begin{align}
{\bigcup ([\TT_\star]\setminus[\TT_{+}])\subseteq \bigcup\RR_\star(\TT_{+})}\quad\text{and}\quad |\RR_\star(\TT_+)|\le \Cref |[\TT_\star]\setminus[\TT_{+}]|.
\end{align}
\end{lemma}
{For the proof of Lemma~\ref{lem:discrete reliability}, we need to introduce a Scott-Zhang type operator.
Let $[\TT_\star]\in[\T]$ and $\set{R_{i,p}|_{[a,b\rangle}}{i=1-p,\dots,N-p}\circ\gamma|_{[a,b\rangle}^{-1}$ be the basis of NURBS of $\XX_\star$, where "$\rangle$" stands for "$)$" if $\Gamma=\partial\Omega$ is closed and for "$]$" if $\Gamma\subsetneqq \partial\Omega$ is open.
Here, $N$ denotes the number of transformed knots $\check{\KK}_\star$ in $(a,b]$.
With the corresponding B-splines there holds $R_{i,p}=B_{i,p} /w$, where $w=\sum_{\ell\in \Z} w_\ell B_{\ell,p}$ is the fixed denominator satisfying $w_{\rm min}\le w\le w_{\rm max}$; see Section~\ref{section:igabem}.
In \cite[Section 2.1.5]{overview}, it is shown that, for $i\in \{1-p,\dots,N-p\}$, there exist dual basis functions $B_{i,p}^*\in L^2(\supp B_{i,p})$ with 
\begin{align}
\int_{\supp B_{i,p}} B_{i,p}^*(t) B_{j,p}(t) dt=\delta_{ij}=\begin{cases} 1\quad \text{if }i=j\\ 0\quad \text{else}\end{cases}
\end{align}
and
\begin{align}\label{eq:dual inequality}
\norm{B_{i,p}^*}{L^2(\supp B_{i,p})}\le (2p+3)9^p {|\supp B_{i,p}|^{-1/2}}.
\end{align}
Define  $R_{i,p}^*:=B_{i,p}^* w$ with the denominator $w$ from before, and $\widehat{R}_{i,p}:=R_{i,p}|_{[a,b\rangle}\circ \gamma|_{[a,b\rangle}^{-1}$.
For $I\subseteq \{1-p,\dots,N-p\}$, we define the following Scott-Zhang type operator
\begin{equation}\label{eq:scotty}
P_{\star,I}:L^2(\Gamma)\to \XX_\star:\psi\mapsto \sum_{i\in I} \Big(\int_{\supp {R}_{i,p}} R_{i,p}^*(t) \psi(\gamma(t)) \,dt \Big)\widehat{R}_{i,p}.
\end{equation}
In \cite[Section 3.1.2]{overview}, a similar operator is considered  for $I= \{1-p,\dots,N-p\}$, and \cite[Proposition 2.2]{overview} proves an analogous version of the following lemma.}
\begin{lemma}\label{lem:scotty}
The Scott-Zhang type operator \eqref{eq:scotty} satisfies the following two properties:
\begin{enumerate}[\rm(i)]
\item \label{item:scotty projection} Local projection property:
For $T\in \TT_\star$ with $\set{i}{T\subseteq \supp \widehat{R}_{i,p}}\subseteq I$ and $\psi\in L^2(\Gamma)$, the inclusion {$\psi|_{\omega_\star^{p}(T)} \in \XX_\star|_{\omega_\star^p(T)}:=\set{\xi|_{\omega_\star^p(T)}}{\xi\in\XX_\star}$} implies $\psi|_T=(P_{\star,I}\psi)|_T$.
\item \label{item:scotty stable} Local $L^2$-stability: 
For $\psi\in L^2(\Gamma)$ and $T\in \TT_\star$, there holds  
\begin{equation*}
\norm{P_{\star,I}(\psi)}{L^2(T)}\leq \Cstb \norm{\psi}{L^2(\omega_\star^{p}(T))},
\end{equation*}
where {$\Cstb$ depends only on $\check{\kappa}_{\max}, p, w_{\max}$, and $\gamma$.}
\end{enumerate}
\end{lemma}
\begin{proof}{
All NURBS basis functions which are non-zero on $T$, have support in $\omega_\star^p(T)$.
With this, {\rm (i)} follows easily  from the definition of $P_{\star,I}$.
For stability {\rm (ii)}, we use $0\leq \widehat{R}_{i,p}\leq 1$ and \eqref{eq:dual inequality} to see
\begin{eqnarray*}
\norm{P_{\star,I}\psi}{L^2(T)}&=&\Big\|\sum_{i\in I} \Big(\int_{\supp {R}_{i,p}} {R}_{i,p}^*(t) \psi(\gamma(t)) \,dt\Big) \widehat{R}_{i,p}\Big\|_{L^2(T)}\\
&\le&\sum_{i\in I\atop |\supp \widehat{R}_{i,p}\cap T|>0} \Big|\int_{\supp {R}_{i,p}} {R}_{i,p}^*(t) \psi(\gamma(t))\, dt\Big| \norm{\widehat{R}_{i,p}}{L^2(T)}\\
&\lesssim& \sum_{i\in I\atop |\supp \widehat{R}_{i,p}\cap T|>0} \norm{{R}_{i,p}^*}{L^2(\supp {R}_{i,p})} {\norm{\psi}{L^2(\supp \widehat{R}_{i,p})}}h_{\star,T}^{1/2}\\  
&\stackrel{\eqref{eq:dual inequality}}{\lesssim}&\sum_{i\in I\atop |\supp \widehat{R}_{i,p}\cap T|>0} \norm{\psi}{L^2(\supp \widehat{R}_{i,p})}\lesssim \norm{\psi}{L^2(\omega_\star^p(T))}.
\end{eqnarray*}
Overall, the hidden constants depend only on $\check{\kappa}_{\max}, p, w_{\max}$, and $\gamma$.}
\end{proof}

\begin{proof}[Proof of Lemma~\ref{lem:discrete reliability}]
We choose 
\begin{align*}
I:=\set{i}{|\supp \widehat{R}_{i,p}\cap \Gamma\setminus \omega_\star^{p}([\TT_\star]\setminus [\TT_{+}])|>0}.
\end{align*}
We prove that
\begin{align}\label{eq:projdiff}
\begin{split}
P_{\star,I}(\Phi_{+}-\Phi_\star)=
\begin{cases}
\Phi_{+}-\Phi_\star &\text{ on }  \Gamma\setminus\omega_\star^{p}( [\TT_\star]\setminus [\TT_{+}]),\\
0&\text { on }\bigcup  ([\TT_\star]\setminus [\TT_{+}]).
\end{cases}
\end{split}
\end{align}
To see this, let  $T\in \TT_\star$ with $T\subseteq \Gamma\setminus\omega_\star^{p}( [\TT_\star]\setminus [\TT_{+}])$ {(up to finitely many points)}.
Then, $\set{i}{T\subseteq \supp \widehat{R}_{i,p}}\subseteq I$.
It holds $\omega_\star^{p}(T)\subseteq \bigcup ([\TT_\star]\cap[\TT_{+}])$.
{This implies that no new knots are inserted in $\omega^p_\star(T)$.
With Lemma \ref{lem:properties for B-splines} \eqref{item:spline basis} it follows $\XX_+|_{\omega_\star^p(T)}=\XX_\star|_{\omega_\star^p(T)}$ and in particular  $(\Phi_{+}-\Phi_\star)|_{\omega_\star^p(T)}\in \XX_\star|_{\omega^p_\star(T)}$}.
Hence Lemma~\ref{lem:scotty} \eqref{item:scotty projection} is applicable and proves $P_{\star,I}(\Phi_+-\Phi_\star)|_T=(\Phi_+-\Phi_\star)|_T$.
For $T\in \TT_\star$ with $T\subseteq \bigcup ([\TT_\star]\setminus [\TT_+])$, the assertion
 follows immediately from the definition of $P_{\star,I}$, since $\widehat{R}_{i,p}|_T=0$ for $i\in I$.

Let $\widetilde{\NN}_\star:=\set{z\in \NN_\star}{z\in \omega_\star^{p}( [\TT_\star]\setminus [\TT_{+}])}$. 
For $z\in \widetilde{\NN}_\star$, let $\varphi_z$ be the $P^1$ hat function, i.e., ${\varphi}_z(z')=\delta_{zz'}$ for all $z'\in\NN_\star$, $\supp ({\varphi}_z)=\omega_\star(z)$, and  $\partial_\Gamma \varphi_z= \mathrm{const.}$ on $T_{z,1}$ and $T_{z,2}$, where $\omega_\star(z)=T_{z,1}\cup T_{z,2}$ with $T_{z,1},T_{z,2}\in \TT_\star$.
Because of Galerkin orthogonality and $\sum_{z\in \widetilde{\NN}_\star} \varphi_z=1$ on $\omega_\star^{p}( [\TT_\star]\setminus [\TT_{+}])$, we see
\begin{align*}
\norm{\Phi_{+}-\Phi_\star}{V}^2&=\dual{f-V\Phi_\star}{(1-P_{\star,I})(\Phi_+-\Phi_\star)}_\Gamma\\&=\Big\langle\sum_{z\in \widetilde{\NN}_\star} \varphi_z (f-V\Phi_\star);(1-P_{\star,I})(\Phi_{+}-\Phi_\star)\Big\rangle_\Gamma.
\end{align*}
We abbreviate $\Sigma:=\sum_{z\in \widetilde{\NN}} \varphi_z (f-V\Phi_\star)$ and estimate with {\rm (M1)},  Lemma~\ref{lem:scotty} \eqref{item:scotty stable} and Proposition~\ref{prop:inverse estimate}
\begin{eqnarray*}
\dual{\Sigma}{P_{\star,I}(\Phi_{+}-\Phi_\star)}&\le& \norm{h_\star^{-1/2}\Sigma}{L^2(\Gamma)}\norm{h_\star^{1/2} P_{\star,I}(\Phi_{+}-\Phi_\star)}{L^2(\Gamma)}\\
&\stackrel{\eqref{eq:projdiff}}{=}&\norm{h_\star^{-1/2} \Sigma}{L^2(\Gamma)} \norm{h_{+}^{1/2} P_{\star,I}(\Phi_{+}-\Phi_\star)}{L^2(\bigcup ([\TT_\star]\cap [\TT_{+}]))} \\
&\stackrel{\text{Lem.}~\ref{lem:scotty}}{\lesssim}& \norm{h_\star^{-1/2} \Sigma}{L^2(\Gamma)} \norm{h_{+}^{1/2} (\Phi_{+}-\Phi_\star)}{L^2(\omega_\star^p( [\TT_\star]\cap [\TT_{+}]))}\\
&\stackrel{\text{Prop.}~\ref{prop:inverse estimate}}{\lesssim}& \norm{h_\star^{-1/2} \Sigma}{L^2(\Gamma)} \norm{\Phi_{+}-\Phi_\star}{V},
\end{eqnarray*}
as well as
\begin{align*}
\dual{\Sigma}{\Phi_{+}-\Phi_\star}\le \norm{\Sigma}{H^{1/2}(\Gamma)} \norm{\Phi_{+}-\Phi_\star}{\widetilde{H}^{-1/2}(\Gamma)}\simeq \norm{\Sigma}{H^{1/2}(\Gamma)}\norm{\Phi_{+}-\Phi_\star}{V}.
\end{align*}
So far, we thus have proved
\begin{align}\label{eq:discrete reliability 1}
\begin{split}
\norm{\Phi_{+}-\Phi_\star}{V}&\le \norm{h_\star^{-1/2}\Sigma}{L^2(\Gamma)} +\norm{\Sigma}{H^{1/2}(\Gamma)}\\
&\le \norm{h_\star^{-1/2}(f-V\Phi_\star)}{L^2(\omega^{p+1}([\TT_\star]\setminus [\TT_{+}]))}+\norm{\Sigma}{H^{1/2}(\Gamma)}.
\end{split}
\end{align}
Next, we use \cite[Lemma 3.4]{igafaermann},   \cite[Lemma 4.5]{resigabem}, and $|\partial_\Gamma\varphi_z|\simeq |\omega_\star(z)|^{-1}$ to estimate
\begin{eqnarray}\label{eq:discrete reliability 2}
&\norm{\Sigma}{H^{1/2}(\Gamma)}^2 &\stackrel{\text{\cite{igafaermann}}}\lesssim\sum_{z\in \NN_\star} |\Sigma|^2_{H^{1/2}(\omega_\star(z))}+\norm{h_\star^{-1/2}\Sigma}{L^2(\Gamma)}^2\notag\\
&\stackrel{\text{\cite{resigabem}}}\lesssim& \sum_{z\in \NN_\star}\norm{h_\star^{1/2} \partial_\Gamma \Sigma}{L^2(\omega_\star(z))}^2+\norm{h_\star^{-1/2}\Sigma}{L^2(\Gamma)}^2\notag\\
&\lesssim& \norm{h_\star^{1/2}\partial_\Gamma \Sigma}{L^2(\Gamma)}^2+\norm{h_\star^{-1/2}\Sigma}{L^2(\Gamma)}^2\notag\\
&\lesssim &\Big\|{h_\star^{1/2} \partial_\Gamma(f-V\Phi_\star)\sum_{z\in \widetilde{\NN}_\star} \varphi_z}\Big\|_{L^2(\Gamma)}^2+\Big\|{h_\star^{1/2}(f-V\Phi_\star) \sum_{z\in \widetilde{\NN}_\star} \partial_\Gamma \varphi_z }\Big\|_{L^2(\Gamma)}^2+\norm{h_\star^{-1/2}\Sigma}{L^2(\Gamma)}^2\notag\\
&\lesssim &\norm{h_\star^{1/2} \partial_\Gamma(f- V\Phi_\star)}{L^2(\omega_\star^{p+1}([\TT_\star]\setminus [\TT_{+}]))}^2+ \norm{h_\star^{-1/2}(f-V\Phi_\star)}{L^2(\omega_\star^{p+1}([\TT_\star]\setminus [\TT_{+}]))}^2.
\end{eqnarray}
It remains to consider the  term $\norm{h_\star^{-1/2}(f-V\Phi_\star)}{L^2(\omega^{p+1}([\TT_\star]\setminus [\TT_{+}]))}$ of \eqref{eq:discrete reliability 1} and  \eqref{eq:discrete reliability 2}.
It holds 
\begin{align}\label{eq:discrete reliability 4}
\norm{h_\star^{-1/2} (f-V\Phi_\star)}{L^2(\omega_\star^{p+1}([\TT_\star]\setminus [\TT_{+}]))}^2=\sum_{ T\in \TT_\star \atop T\subseteq \omega_\star^{p+1}([\TT_\star]\setminus [\TT_{+}])}\norm{h_\star^{-1/2}(f-V\Phi_\star)}{L^2(T)}^2.
\end{align}
{For any $T\in \TT_\star$, there is a function $\psi_T\in\XX_\star$ with connected support, $T\subseteq \supp(\psi_T)\subseteq \omega_\star^{\lceil p/2\rceil}(T)$ and $\norm{1-\psi_T}{L^2(\supp(\psi_T))}^2\le q\, |\supp(\psi_T)|$ with some $q\,\in (0,1)$ which depends only on $\check{\kappa}_{\rm max}, \gamma, p, w_{\rm min}$, and  $w_{\rm max}$; see \cite[{\rm(A1)--\rm(A2)}, Theorem~4.4]{igafaermann}.}
We use some Poincar\'e inequality (see, e.g., \cite[Lemma 2.5]{faermann2d}) to see
\begin{align}\label{eq:discrete reliability 3}
\begin{split}
&\norm{f-V\Phi_\star}{L^2(\supp {\psi}_T)}^2\\
&\quad \le {\frac {|\supp(\psi_T)|^2}{2}} \norm{\partial_\Gamma (f-V\Phi_\star)}{L^2(\supp(\psi_T))}^2+\frac{1}{|\supp(\psi_T)|} \left|\int_{\supp(\psi_T)} (f-V\Phi_\star)(x) \,dx\right|^2.
\end{split}
\end{align}
The Galerkin orthogonality proves
\begin{align*}
\left |\int_{\supp(\psi_T)} (f-V\Phi_\star)(x) \,dx\right|^2 &= \left|\int_{\supp(\psi_T)} (f-V\Phi_\star)(x)(1-\psi_T(x))\,dx\right|^2\\
&\le\norm{f-V\Phi_\star}{L^2(\supp(\psi_T))}^2 q\, |\supp(\psi_T)|.
\end{align*}
Using \eqref{eq:discrete reliability 3}, we therefore get 
\begin{align*}
\norm{f-V\Phi_\star}{L^2(\supp(\psi_T))}^2\le \frac{|\supp(\psi_T)|^2}{2} \norm{\partial_\Gamma(f-V\Phi_\star)}{L^2(\supp(\psi_T))}^2 +q\,\norm{f-V\Phi_\star}{L^2(\supp(\psi_T))}^2,
\end{align*}
which implies 
\begin{align*}
\norm{f-V\Phi_\star}{L^2(T)}^2\lesssim {h_{\star,T}^2}\norm{\partial_\Gamma(f-V\Phi_\star)}{L^2(\omega^{\lceil p/2 \rceil}(T))}^2.
\end{align*}
Hence, we are led to
\begin{align*}
\norm{h_\star^{-1/2} (f-V\Phi_\star)}{L^2(\omega^{p+1}([\TT_\star]\setminus [\TT_{+}]))}^2\lesssim \norm{h_\star^{1/2} \partial_\Gamma(f-V\Phi_\star)}{L^2(\omega^{\lceil p/2\rceil+ p+1}([\TT_\star]\setminus [\TT_{+}]))}^2.
\end{align*}
With 
\begin{align*}
\mathcal{R}_\star(\TT_{+}):=\set{T\in\TT_\star}{T\subseteq \omega^{\lceil p/2\rceil +p+1}([\TT_\star]\setminus [\TT_{+}])},
\end{align*}
 we therefore conclude the proof.
\end{proof}

{Since we use a different mesh-refinement strategy, we cannot directly cite the following lemma from \cite{axioms}. 
However, we may essentially follow the proof of \cite[Proposition 4.12]{axioms} verbatim.
Details are left to the reader.

\begin{lemma}[optimality of D\"orfler marking]\label{lem:Doerfler optimal}
Define
\begin{align}
\overline\theta_{\rm opt}:=(1+\Cstab^2\Cdrel^2)^{-1}.
\end{align}
For all $0<\overline\theta<\overline\theta_{\rm opt}$ there is some $0<\qdoe<1$ such that for all refinements $[\TT_{+}]\in\refine([\TT_\star])$ of $[\TT_\star]\in[\T]$ the following implication holds true
\begin{align}
\rho_{+}^2\le \qdoe\rho_\star^2\quad\Longrightarrow \quad\overline\theta \rho_\star^2\le \sum_{T\in \RR_\star(\TT_+)} \rho_\star(T)^2.
\end{align}
The constant $\qdoe$ depends only on $\overline\theta$ and the constants $\Cstab$ of Lemma~\ref{lem:stability rho} and $\Cdrel$ of Lemma~\ref{lem:discrete reliability}. \hfill $\square$
\end{lemma}


The next lemma reads similarly as \cite[Lemma 3.4]{axioms}. 
{Since we use a different mesh-refinement strategy and our estimator $\rho$ does not satisfy the reduction axiom {\rm (A2)}, we cannot directly cite the result. 
However, the idea of the proof is the same. 
Indeed, one only needs a weaker version of the mentioned axiom.}

\begin{lemma}[quasi-monotonicity of $\rho$]\label{lem:rho quasi-monotone}
For all refinements $[\TT_{+}]\in\refine([\TT_\star])$ of $[\TT_\star]\in [\T],$ there holds 
\begin{align}
\rho_{+}^2\leq \Cmon \rho_\star^2,
\end{align}
where $\Cmon>0$ depends only on {the parametrisation $\gamma$ and the constants $\CinvV$ of Proposition~\ref{prop:inverse estimate} and $\Cdrel$ of Lemma~\ref{lem:discrete reliability}.}
\end{lemma}
\begin{proof}

We split the estimator as follows 
\begin{align*}
\rho_{+}^2=\sum_{T\in \TT_{+}\setminus \TT_\star} \rho_{+}(T)^2+\sum_{T\in \TT_\star\cap \TT_{+}} \rho_+(T)^2.
\end{align*}
For the first sum, we {use \eqref{eq:prestability}}, $\bigcup (\TT_{+}\setminus \TT_\star)=\bigcup (\TT_\star\setminus \TT_{+})$, and $\check{h}_{+}\le \check{h}_\star$ to estimate
\begin{align*}
\sum_{T\in \TT_{+}\setminus \TT_\star}\rho_{+}(T)^2&=\norm{\check{h}_{+}^{1/2} \partial_\Gamma (f-V\Phi_{+})}{L^2(\bigcup(\TT_{+}\setminus \TT_\star))}^2\\
&{\lesssim \left( \norm{\Phi_+-\Phi_\star}{\H^{-1/2}(\Gamma)}+\norm{\check{h}_\star^{1/2} \partial_\Gamma (f-V\Phi_\star)}{L^2(\bigcup(\TT_{\star}\setminus \TT_+))}\right)^2}\\
&\le 2 \norm{\Phi_{+}-\Phi_\star}{\H^{-1/2}(\Gamma)}^2 +{2}\sum_{T\in\TT_\star\setminus \TT_{+}} \rho_\star(T)^2 
\end{align*}
For the second sum, we use Lemma~\ref{lem:stability rho} to see
\begin{align*}
\sum_{T\in \TT_\star\cap \TT_{+}} \rho_{+}(T)^2\le 2\sum_{T\in \TT_\star\cap\TT_{+}} \rho_\star(T)^2+2\Cstab^2 \norm{\Phi_\star-\Phi_{+}}{\H^{-1/2}(\Gamma)}^2.
\end{align*}
We end up with
\begin{align*}
\rho_{+}^2\lesssim \norm{\Phi_{+}-\Phi_\star}{\H^{-1/2}(\Gamma)}^2 + {\rho_\star^2}.
\end{align*}
Lemma~\ref{lem:discrete reliability} concludes the proof.
\end{proof}
The optimality in Theorem~\ref{thm:main} essentially follows from the following lemma. 
It was inspired by an analogous version from \cite[Lemma 4.14]{axioms}.
\begin{lemma}\label{lem:main lemma}
Suppose that $\phi\in \A_s$ for some $s>0$. Then, for all $0<\overline \theta<\overline\theta_{\rm opt}$ there exist constants $\Cone,\Ctwo>0$ such that for all meshes $[\TT_\star]\in [\T]$ there exists some refinement $[\TT_{+}]\in \refine([\TT_\star])$ such that the corresponding set $\RR_\star(\TT_+)\subseteq \TT_\star$ from Lemma~\ref{lem:discrete reliability}  satisfies
\begin{align}\label{eq:CC}
|\RR_\star(\TT_+)|\le \Cone\Ctwo^{1/s}\norm{\phi}{\A_s}^{1/s}\rho_\star^{-1/s},
\end{align}
and
\begin{align}\label{eq:Doerfler for R}
\overline \theta \rho_\star^2\le \sum_{T\in \RR_\star(\TT_+)}\rho_\star(T)^2.
\end{align}
{With the constants $\Cdrel, \Cmon$, and $\qdoe$ from Lemma~\ref{lem:discrete reliability},~\ref{lem:Doerfler optimal} and~\ref{lem:rho quasi-monotone},} it holds $\Cone=2\Cdrel$ and {$\Ctwo=(\Cmon \qdoe^{-1})^{1/2}$}.
\end{lemma}
\begin{proof}
We set $\alpha:=\Cmon^{-1} \qdoe$ with the constants of Lemma~\ref{lem:Doerfler optimal} and Lemma~\ref{lem:rho quasi-monotone}, and $\delta^2:=\alpha \rho_\star^2$.

\noindent
{\bf Step 1:}
There exists $[\TT_\delta]\in [\T]$ with
\begin{align*}
\rho_\delta\le 
|\KK_\delta|-|\KK_0|\le \norm{\phi}{\A_s}^{1/s}\delta^{-1/s}.
\end{align*}
Let $N\in \N_0$ be minimal with $(N+1)^{-s}\norm{\phi}{\A_s}\le \delta$.
If $N=0$, then $\rho_0\le \norm{\phi}{\A_s}\le \delta$ and we can choose $[\TT_\delta]=[\TT_0]$.
If $N>0$, minimality of $N$ implies $N^{-s}\norm{\phi}{\A_s}>\delta$ or equivalently $N+1\le \norm{\phi}{\A_s}^{1/s}\delta^{-1/s}$.
Now we choose $[\TT_\delta]\in [\T_N]$ such that 
\begin{align*}
\rho_\delta=\min_{[\TT_\star]\in[\T_N]} \rho_\star
\end{align*}
and see
\begin{align*}
\rho_\delta\le (N+1)^{-s}\norm{\phi}{\A_s}\le \delta.
\end{align*}
\noindent
{\bf Step 2:}
We consider the overlay $[\TT_{+}]:= [\TT_\star]\oplus [\TT_\delta]$ of {\rm (M2)}.
Quasi-monotonicity shows 
\begin{align}\label{eq:Doerfler praem}
\rho_{+}^2\le \Cmon \rho_\delta^2\le \Cmon \delta^2=\qdoe\rho_\star^2.
\end{align}
\noindent
{\bf Step 3:}
Finally, the assumptions on the refinement strategy are used.
The overlay estimate and Step 1 give
\begin{align*}
|\KK_{+}|-|\KK_\star|\le (|\KK_\delta|+|\KK_\star|-|\KK_0|)-|\KK_\star|=|\KK_\delta|-|\KK_0|\le \norm{\phi}{\A_s}^{1/s} \delta^{-1/s}.
\end{align*}
Lemma~\ref{lem:discrete reliability} and {\rm (M3)} show 
\begin{align*}
|\RR_\star(\TT_+)|\le \Cdrel| [\TT_\star]\setminus[\TT_{+}]|\le 2 \Cdrel (|\KK_{+}|-|\KK_\star|).
\end{align*}
Combining the last two estimates, we end up with
\begin{align*}
|\RR_\star(\TT_+)|\le 2\Cdrel\norm{\phi}{\A_s}^{1/s}\alpha^{-1/2s}\rho_\star^{-1/s},
\end{align*}
This proves \eqref{eq:CC} with $\Cone=2\Cdrel$ and $\Ctwo=\alpha^{-1/2}$.
By \eqref{eq:Doerfler praem} we can apply Lemma~\ref{lem:Doerfler optimal} and see \eqref{eq:Doerfler for R}.
\end{proof}
So far, we have only considered the auxiliary estimator $\rho_\star$
In particualar, we did not use Algorithm~\ref{the algorithm}, but only the refinement strategy $\refine(\cdot)$ itself.
For the proof of optimal convergence  \eqref{eq:optimal}, we proceed similarly as in \cite[Theorem 8.4 (ii)]{axioms}.
\begin{proof}[Proof of \eqref{eq:optimal}]
{Due to \eqref{eq:loceq}}, there is  a constant $C\ge 1$ which depends only  on $\check{\kappa}_{\max}$ with $\mu_\star^2\le C\rho_\star ^2$ for all $[\TT_\star] \in [\T]$.
We set $\theta_{\rm opt}:= {\overline\theta_{\rm opt}}/{C}$ and $\overline \theta:=C\,\theta$ and suppose that $\theta$ is sufficiently small, namely, $\theta<\theta_{\rm opt}$ and hence $\overline{\theta}<\overline{\theta}_{\rm opt}$.
Let $\ell\in \N_0$ and $j\leq \ell$.
Choose a refinement $[\TT_{+}]$ of $[\TT_j]$ as in Lemma~\ref{lem:main lemma}.
In particular, the set $\RR_j(\TT_+)$ satisfies the D\"orfler marking \eqref{eq:Doerfler praem}.
According to \eqref{eq:loceq}, this implies 
\begin{align*}
\theta\mu_j^2\le \overline \theta \rho_j^2\le \sum_{T\in \RR_j(\TT_+)}\rho_j(T)^2 \le \sum_{z \in \NN_j \cap \bigcup \RR_j(\TT_+)} \mu_j(z)^2,
\end{align*}
{i.e., the set $\NN_j \cap \bigcup \RR_j(\TT_+)$ satisfies the D\"orfler marking \eqref{eq:Doerfler} from Algorithm~\ref{the algorithm}.
Since the chosen set  $\MM_j$ of Algorithm~\ref{the algorithm} has essentially minimal cardinality, we see with \eqref{eq:CC} that}
\begin{align*}
|\MM_j|\le \Cmark|\NN_j\cap \bigcup \RR_j(\TT_+)|\leq 2\Cmark|\RR_j(\TT_+)|\le 2 \Cmark \Cone\Ctwo^{1/s}\norm{\phi}{\A_s}^{1/s}\rho_j^{-1/s}
\end{align*}
With the mesh-closure estimate of {\rm (M2)}, we get
\begin{align*}
|\KK_\ell|-|\KK_0|\le \Cmesh \sum_{j=0}^{\ell-1} |\MM_j|&\le 2\Cmark\Cmesh \Cone\Ctwo^{1/s} \norm{\phi}{\A_s}^{1/s} \sum_{j=0}^{\ell-1} \rho_j^{-1/s}\\&\le 2\Cmark\Cmesh \Cone\Ctwo^{1/s} C^{1/s}\norm{\phi}{\A_s}^{1/s} \sum_{j=0}^{\ell-1} \mu_j^{-1/s}.
\end{align*}
Linear convergence \eqref{eq:R-linear} shows
\begin{align*}
\mu_\ell\le \Clin\qlin^{\ell-j} \mu_j\quad\text{for all }j=0,\dots,\ell.
\end{align*}
Hence,
\begin{align*}
|\KK_\ell|-|\KK_0|&\le 2\Cmark\Cmesh \Cone\Ctwo^{1/s} C^{1/s}\norm{\phi}{\A_s}^{1/s} \sum_{j=0}^{\ell-1} \mu_j^{-1/s}\\&\le 2\Cmark\Cmesh \Cone(\Ctwo\Clin C)^{1/s} \norm{\phi}{\A_s}^{1/s}\mu_\ell^{-1/s}\sum_{j=0}^{\ell-1} (\qlin^{1/s})^{\ell-j}\\
&\le  (\Ctwo\Clin C)^{1/s} \frac{2\Cmark\Cmesh \Cone}{1-\qlin^{1/s}} \norm{\phi}{\A_s}^{1/s}\mu_\ell^{-1/s}.
\end{align*}
This  concludes the proof.
\end{proof}


\section{Proof of Theorem~\ref{thm:faermann}, plain convergence~\eqref{eq:Faermann converges}}
\label{section:faermann}
%
%
To prove convergence of Algorithm~\ref{the algorithm} driven by the Faermann estimators $\eta_\ell$, we apply an abstract result of \cite[Section 2]{mitscha} which is recalled in the following:
Let $\HH$ be a Hilbert space with dual space $\HH^*$ and $V:\HH\to\HH^*$  be a linear elliptic operator and $f\in \HH^*$.
Let $(\XX_\ell(f))_{\ell\in\N_0}$ be a sequence of finite dimensional nested subspaces of $\HH$, i.e., $\XX_\ell(f)\subseteq\XX_{\ell+1}(f)$, with Galerkin approximations $\Phi_\ell(f)\in\XX_\ell(f)$ for the equation $V\phi=f$.
Further, let $(\NN_\ell(f))_{\ell\in\N_0}$ be a sequence of arbitrary finite sets and
\begin{align*}
\eta_\ell(f):=\eta_\ell(f,\NN_\ell(f))\text{ with }\eta_\ell(f,\EE_\ell):=\Big(\sum_{z\in\EE_\ell} \eta_\ell(f,z)^2\Big)^{1/2}<\infty \text{ for all }\EE_\ell\subseteq\NN_\ell(f)
\end{align*}
some heuristical error estimator, where we only require $\eta_\ell(f,z)\ge 0$ for each $z\in\NN_\ell(f)$.
Let $(\MM_\ell(f))_{\ell\in\N_0}$ be a sequence of marked elements with $\MM_\ell(f)\subseteq\NN_\ell(f)$ which satisfies the D\"orfler marking, i.e.,
\begin{align*}
\theta\eta_\ell(f)^2\le \eta_\ell(f,\MM_\ell(f))^2.
\end{align*}
Additionally let
\begin{align*}
\widetilde\rho_\ell(f):=\widetilde\rho_\ell(f,\NN_\ell(f))\text{ with }\widetilde\rho_\ell(f,\EE_\ell):=\Big(\sum_{z\in\EE_\ell} \widetilde\rho_\ell(f,z)^2\Big)^{1/2}<\infty \text{ for all }\EE_\ell\subseteq\NN_\ell(f)
\end{align*}
be an auxiliary error estimator with local contributions $\widetilde\rho_\ell(f,z)\ge 0$.
Then, there holds the following convergence result.
\begin{lemma}\label{lem:mitscha}
Suppose that  $D\subseteq \HH^*$ is a dense subset of $\HH^*$ such that for all $f\in D$ and all $\ell\in\N_0$ there is a set $\RR_\ell(f)\subseteq \MM_\ell(f)$ such that the following assumptions {\rm(A1)--(A3)} hold:
\begin{itemize}
\item[\rm(A1)] $\eta_\ell(f)$ is a local lower bound of $\widetilde\rho_\ell(f)$:
There is a constant $C_1>0$ such that
\begin{align*}
\eta_\ell(f,\MM_\ell(f))\le C_1 \widetilde\rho_\ell(f,\RR_\ell(f))\quad\text{for all }\ell\in\N_0.
\end{align*}
\item[\rm(A2)] $\widetilde\rho_\ell(f)$ is contractive on $\RR_\ell(f)$: There is a constant $C_2$ such that for all $\ell\in\N_0, m\in \N$ and all $\delta>0$, it holds
\begin{align*}
C_2^{-1}\widetilde\rho_\ell(f,\RR_\ell(f))^2\le \widetilde\rho_\ell(f)^2-\frac{1}{1+\delta}\, \widetilde\rho_{\ell+m}(f)^2+(1+\delta^{-1})C_2\norm{\Phi_{\ell+m}(f)-\Phi_\ell(f)}{\HH}^2.
\end{align*}
\end{itemize}
In addition, we suppose for all $f\in \HH^*$ validity of:
\begin{itemize}
\item[\rm(A3)]
$\eta_\ell$ is stable on $\MM_\ell(f)$ with respect to $f$: there is a constant $C_3>0$ such that for all $\ell\in\N_0$ and $f'\in\HH^*$, it holds
\begin{align*}
|\eta_\ell(f,\MM_\ell(f))-\eta_\ell(f',\MM_\ell(f))|\le C_3 \norm{f-f'}{\HH^*}.
\end{align*}
\end{itemize}
Then, there holds convergence $\lim_{\ell\to\infty}\eta_\ell=0$ for all $f\in\HH^*$.
\end{lemma}
\begin{proof}[Proof of plain convergence~\eqref{eq:Faermann converges}] 
We choose $\HH=\H^{-1/2}(\Gamma)$, $\HH^*=H^{1/2}(\Gamma)$, $V$ the weakly-singular integral operator~\eqref{eq:strong}.
Moreover, Algorithm~\ref{the algorithm} generates the transformed NURBS spaces $\XX_\ell(f)$, the set of nodes $\NN_\ell(f)$, the Faermann estimator $\eta_\ell(f)$ and the set of marked nodes $\MM_\ell(f)$. 
We 
use  the mesh-size function $\widetilde{h}_\ell$ of Proposition~\ref{lem:mesh function} to define 
\begin{align}
\widetilde\rho_\ell(f,z):=\norm{\widetilde{h}_\ell^{1/2}\partial_\Gamma(f-V\Phi_\ell)}{L^2(\omega(z))}\quad\text{for all }z\in\NN_\ell,
\end{align}
if $f$ is in the dense set $D=H^1(\Gamma)$.
We aim to apply Lemma~\ref{lem:mitscha} and show in the following that{\rm(A1)--(A2)} hold for all $f\in H^1(\Gamma)$ even with $\RR_\ell(f)=\MM_\ell(f)$ and that  {\rm (A3)} holds for all $f\in H^{1/2}(\Gamma)$.
Then, Lemma~\ref{lem:mitscha} shows convergence~\eqref{eq:Faermann converges} of the Faermann estimator.

 {\rm (A1)} of Lemma~\ref{lem:mitscha} follows immediately from \cite[Theorem 4.3]{resigabem}, where the constant $C_1$ depends only on $\check{\kappa}_{\max}, p$, and $\gamma$. 
 
 {\rm(A3)} can be proved exactly as in \cite[Section~2.4]{mitscha} as $\eta_\ell$ is efficient (see \cite[Theorem 3.1]{igafaermann}) and has a semi-norm structure. 
The constant $C_3$ depends only on $\Gamma$.

The only challenging part is the proof of {\rm(A2)} for fixed $f\in H^1(\Gamma)$.
We proceed similarly as in the proof of \cite[Theorem 3.1]{mitscha}.
In the following, we  skip the dependence of $f$.
The heart of the matter are the estimates $\widetilde{h}_{\ell+1}\le \qctr \widetilde h_{\ell}$ on $\omega_\ell(\MM_\ell)$ and $\widetilde h_{\ell+1}\le \widetilde h_\ell$ on $\Gamma$, which follow from Proposition~\ref{lem:mesh function} and 
 \begin{align*}
 \MM_\ell\subset \bigcup([\TT_\ell]\setminus[\TT_{\ell+1}]).
 \end{align*}
This shows 
\begin{align*}
\widetilde h_\ell-\widetilde h_{\ell+m}\ge \widetilde h_\ell-\widetilde h_{\ell+1}\ge (1-\qctr)\widetilde h_\ell \,\chi_{\omega_\ell(\MM_\ell)}\quad\text{for all }\ell\in\N_0 \text{ and }m\in\N. 
\end{align*}
Hence, the estimator $\widetilde\rho_\ell$ satisfies
\begin{align*}
(1-\qctr)\,\widetilde\rho_\ell(\MM_\ell)^2/2&\le(1-\qctr)\int_{\omega_\ell(\MM_\ell)}\widetilde h_\ell\,|\partial_\Gamma (f-V\Phi_\ell)|^2 dx\\
&\le \int_\Gamma \widetilde h_\ell \,|\partial_\Gamma(f-V\Phi_\ell)|^2-\int_\Gamma\widetilde h_{\ell+m} \,|\partial_\Gamma(f-V\Phi_\ell)|^2 dx\\
&=\norm{\widetilde h_\ell^{\,1/2}\partial_\Gamma(f-V\Phi_\ell)}{L^2(\Gamma)}^2-\norm{\widetilde h_{\ell+m}^{\,1/2}\partial_\Gamma(f-V\Phi_{\ell})}{L^2(\Gamma)}^2.
\end{align*}
Here, the factor $1/2$ on the left-hand side stems from the fact that each node patch consists (generically) of two elements. 
This fact also shows $\norm{\widetilde h_\ell^{\,1/2}\partial_\Gamma(f-V\Phi_\ell)}{L^2(\Gamma)}^2=\widetilde\rho_\ell^{\,2}/2$.
The Young inequality $(c+d)^2\le (1+\delta)c^2+(1+\delta^{-1})d^2$ for $c,d\ge 0$, together with the triangle inequality shows
\begin{align*}
(1-\qctr)\,\widetilde\rho_\ell(\MM_\ell)^2/2\le\widetilde\rho_\ell^{\,2}/2-\frac{1}{1+\delta} \widetilde\rho_{\ell+m}^{\,2}/2+\frac{1+\delta^{-1}}{1+\delta}\norm{\widetilde h_{\ell+m}^{1/2}\partial_\Gamma V(\Phi_\ell-\Phi_{\ell+m})}{L^2(\Gamma)}^2.
\end{align*}
Finally, we use Proposition~\ref{prop:invest} and see
\begin{align*}
\norm{\widetilde h_{\ell+m}^{1/2}\partial_\Gamma V(\Phi_\ell-\Phi_{\ell+m})}{L^2(\Gamma)}\le\widetilde C_{\rm inv} \norm{\Phi_\ell-\Phi_{\ell+m}}{\H^{-1/2}(\Gamma)}.
\end{align*}
with a constant $\widetilde C_{\rm inv}>0$ which depends only on
$\Cinv$ and $h_\star\simeq \widetilde h_\star$.
This yields
\begin{align*}
(1-\qctr)\,\widetilde\rho_\ell(\MM_\ell)^2/2\le \widetilde\rho_\ell^{\,2}/2-\frac{1}{1+\delta}\,\widetilde\rho_{\ell+m}^{\,2}/2+\frac{1+\delta^{-1}}{1+\delta}\widetilde  C_{\rm inv}^2\norm{\Phi_\ell-\Phi_{\ell+m}}{\H^{-1/2}(\Gamma)}^2.
\end{align*}
{and concludes the proof of {\rm (A2)} with $C_2=\max(\frac{1}{1-\qctr},2 C_{\rm inv}^2)$}.
\end{proof}


\bigskip

\noindent
{\bf Acknowledgement.}
The authors acknowledge support through the Austrian Science 
Fund (FWF) under grant P27005 \emph{Optimal adaptivity for BEM and FEM-BEM coupling}. 
In addition, DP,  MF, and GG are supported through the FWF doctoral school \emph{Nonlinear PDEs} 
funded under grant W1245.
\bibliographystyle{alpha}
\bibliography{literature}

\end{document}